\documentclass[11pt]{amsart}

\usepackage[T1]{fontenc}
\usepackage[utf8]{inputenc}

\usepackage[english]{babel}
\usepackage{geometry}
\usepackage{graphicx}

\usepackage{xcolor}
\usepackage{amsmath,amssymb,amsthm,amsfonts,stmaryrd,mathtools}
\usepackage{bbm}
\usepackage{eurosym}
\usepackage{array,dcolumn,mathdots,multirow}
\usepackage{rotating}
\usepackage{fancyvrb}
\usepackage[activate={true,nocompatibility},spacing,kerning]{microtype}

\graphicspath{%
{./Images/}
}

\definecolor{darkgreen}{rgb}{0,0.5,0}
\definecolor{darkred}{rgb}{0.5,0,0}
\definecolor{darkblue}{rgb}{0.004,0.396,0.741}
\definecolor{gcolor}{rgb}{0.004,0.396,0.741}	
\definecolor{warning}{rgb}{1.0,0.6,0.0}	
\definecolor{gray}{rgb}{0.6,0.6,0.6}

\usepackage[pdftex,
            colorlinks,
            hyperfootnotes=false,
            pdffitwindow=true,
            plainpages=false,
            pdfpagelabels=true,
            pdfpagemode=UseOutlines,
            pdfpagelayout=TwoPageRight,
            pdftitle={Asymptotic Expansions at the Soft Edge II: Level Densities},
            pdfauthor={Folkmar Bornemann, TU Munich},
            hyperindex,
            setpagesize=false,
            filecolor=purple,
            urlcolor=darkred,
            citecolor=gcolor,
            linkcolor=gcolor]{hyperref}
\usepackage{remreset}
\usepackage{bm}

\usepackage{listings}
\definecolor{mygreen}{rgb}{0,0.6,0}
\definecolor{mygray}{rgb}{0.5,0.5,0.5}
\definecolor{lightgray}{rgb}{0.925,0.925,0.925}
\definecolor{mymauve}{rgb}{0.58,0,0.82}
\definecolor{gcolor}{rgb}{0.004,0.396,0.741}
\usepackage{caption}
\theoremstyle{plain}

\newtheorem{theorem}{Theorem}[section]

\newtheorem{corollary}{Corollary}[section]
\theoremstyle{definition}

\theoremstyle{remark}
\newtheorem{remark}{Remark}[section]

\numberwithin{equation}{section}

\input{macros.sty}
\VerbatimFootnotes
\makeindex
\begin{document}

\title[Asymptotic Expansions at the Soft Edge II: Level Densities]{Asymptotic Expansions of Gaussian and Laguerre Ensembles at the Soft Edge II: Level Densities}
\author{Folkmar Bornemann}
\address{Department of Mathematics, Technical University of Munich, 80290 Munich, Germany}
\email{bornemann@tum.de}


\begin{abstract}
We continue our work \cite{B25-1} on asymptotic expansions at the soft edge for the classical $n$-dimensional  Gaussian and Laguerre random matrix ensembles. By revisiting the construction of the associated skew-orthogonal polynomials in terms of wave functions, we obtain concise expressions for the level densities that are well suited for proving asymptotic expansions in powers of a certain parameter $h \asymp n^{-2/3}$. In the unitary case, the expansion for the level density can be used to reconstruct the first correction term in an established asymptotic expansion of the associated generating function. In the orthogonal and symplectic cases, we can even reconstruct the conjectured first and second correction terms.
\end{abstract}
\subjclass[2020]{}

\maketitle
\setcounter{tocdepth}{1}

\section{Introduction}\label{sect:intro}

An important statistic of a point process with $n$ levels $x_1,\ldots,x_n$ on an interval $\Omega\subset\R$ is provided by the number distribution $N(x)$ and its expected value,
\begin{subequations}
\begin{equation}
\bar N(x) = \E N(x), \quad N(x) = \# \{j: x_j \leq x\}.
\end{equation}
If the levels of the point process possess a joint probability density $P(x_1,\ldots,x_n)$, the level density, or one-point function, is defined as
\begin{equation}\label{eq:density}
\rho(x) = \bar{N}'(x) = n\int_\Omega\cdots\int_\Omega P(x,x_2,\ldots,x_n)\,dx_2\cdots \,dx_n;
\end{equation}
it satisfies then
\begin{equation}\label{eq:mass}
\int_\Omega\rho(x)\,dx = n.
\end{equation}
\end{subequations}
In this paper, we study the level densities of the orthogonal ($\beta=1$), unitary ($\beta=2$), and symplectic ($\beta=4$) Gaussian and Laguerre random matrix ensembles of dimension $n$, and their asymptotic expansions for the scaling limit at the soft edge as $n\to\infty$. The levels (eigenvalues) of the Gaussian ensembles constitute a point process on $\Omega=(-\infty,\infty)$, while those of the Laguerre ensembles constitute one on $\Omega=(0,\infty)$; the joint probability densities and their parametrizations are recalled in Section~\ref{sect:ensembles}. We write
\[
N_{\beta,n}, \quad P_{\beta,n}, \quad \rho_{\beta,n}, \quad \ldots
\]
to signify the dependence on $\beta$ and $n$; adding a certain additional parameter $p$ for the Laguerre ensembles if necessary. We refer to the index consistency rule of Section~\ref{sect:index} on why the Laguerre parameter $p$ can largely be omitted from the notation and why we can cast the Gaussian cases as the case $p=\infty$ of the Laguerre ones.

\subsection{Concise expressions}
Since the unitary ensembles constitute a determinantal point process with a correlation kernel $K_n(x,y)$, their level densities are known since the pioneering work of Dyson in the 1960s to be
\[
\rho_{2,n}(x) = K_n(x,x);
\]
expressible in terms of orthogonal polynomials. Using the theory of skew-orthogonal polynomials and the quaternionic determinants of Dyson, Mahoux and Mehta \cite{MR1190440} found expression for $\rho_{1,n}$ and $\rho_{4,n}$ -- which are, however, not particularly inviting for asymptotic analysis. Instead, we follow the approach of Tracy and Widom \cite{MR1657844} to that theory, representing the gap probabilities in terms of Fredholm determinants of operators with matrix kernels. By revisiting the construction of the skew-orthogonal polynomials in terms of certain associated wave functions $\psi^\sharp_n$ and denoting their antiderivatives by
\[
\Psi^\sharp_n(x) = -\int_x^\infty \psi_n^\sharp(t)\,dt,
\]
we find in Section~\ref{sect:densities} the concise expressions
\begin{equation}\label{eq:rho14Intro}
\begin{aligned}
\rho_{1,2n+1} &= \rho_{2,2n+1} + \tfrac{1}{2}\psi_{2n}^\sharp\cdot\big(1 + \Psi_{2n+1}^\sharp\big),\\*[2mm]
2\rho_{4,n} &=  \rho_{2,2n+1} + \tfrac{1}{2}\psi_{2n}^\sharp\cdot\big(0 + \Psi_{2n+1}^\sharp\big).
\end{aligned}
\end{equation}
For $\beta=1$, the same expression holds for even dimension with accordingly adjusted indices. Still, we presented it here in the form \eqref{eq:rho14Intro} to show the striking duality between the orthogonal and symplectic cases; a probabilistic proof of that relation is given in Appendix~\ref{app:probwave}.  

\subsection{Scaling limits at the soft edge}
It is known that the largest level (the so-called soft edge of the point process) of the Gaussian and Laguerre ensembles is asymptotically equal, almost surely, to a certain quantity $\mu_n$ when $n\to\infty$. When zooming into the vicinity of the largest level with a certain scale $\sigma_n$, Forrester \cite[§§3.1/3.5]{MR1236195} established\footnote{Here, and unless stated otherwise in other references to the existing literature, the Laguerre parameter $p$ is subject to $q=p-n>-1$ being bounded from above.} that
\begin{subequations}\label{eq:leadingorder}
\begin{equation}
\lim_{n\to\infty} \rho_{2,n}(x)\frac{dx}{ds} \Big|_{x=\mu_n + \sigma_n s} = - s \Ai(s)^2 + \Ai'(s)^2,
\end{equation}
based on the Plancherel--Rotach type asymptotics  \cite[Eqs.~(8.22.11/14)]{MR0372517} of the Hermite and Laguerre polynomials. Applied to the wave functions $\psi_n^\sharp$, the Plancherel--Rotach type asymptotics takes the form
\[
\psi_{n}^\sharp(x)\frac{dx}{dt} \Big|_{x=\mu_n + \sigma_n s} = \Ai(s) + n^{-1/3} O(e^{-s}),\quad 
\Psi_{n}^\sharp(x)\Big|_{x=\mu_n + \sigma_n s} = \AI_0(s) + n^{-1/3} O(e^{-s});
\]
cf. Appendix~\ref{eq:PlancherelRotach}. Here, we write 
\[
\AI_\nu(x) = \nu -\int_x^{\infty} \Ai(t)\,dt
\]
to designate the antiderivative of $\Ai$ satisfying $\AI_\nu(\infty) = \nu$. Thus, the relations \eqref{eq:rho14Intro} imply
\begin{equation}\label{eq:rho4double}
\begin{aligned}
\lim_{n\to\infty} \rho_{1,n}(x)\frac{dx}{ds} \Big|_{x=\mu_n + \sigma_n s} &= - s \Ai(s)^2 + \Ai'(s)^2 + \frac{1}{2}\Ai(s)\AI_1(s),\\*[2mm]
\lim_{n\to\infty} 2\rho_{4,n}(x)\frac{dx}{ds} \Big|_{x=\mu_{2n} + \sigma_{2n} s}&= - s \Ai(s)^2 + \Ai'(s)^2 + \frac{1}{2}\Ai(s)\AI_0(s).
\end{aligned}
\end{equation}
\end{subequations}
These limits were first established by Forrester, Nagao, and Honner \cite[Eqs.~(5.44/45)]{MR1707162}. 

\subsection{Convergence rates and first correction terms}
Based on the error estimates of the Plancherel--Rotach type asymptotics, the scaling limits exhibit a convergence rate of order $O(n^{-1/3})$. It was later observed by  Garoni, Forrester, and Frankel \cite[Proposition 2]{MR2178598} for GUE, and for LUE by Choup \cite[Corollary 3.7]{MR2178598} with a certain adjustment of $\mu_n$ to the Laguerre parameter, that the error is actually of order $O(n^{-2/3})$ -- moreover, the functional form of that term up to $O(n^{-1})$ can be found there. Forrester and Trinh \cite[Remark~6(b)]{MR3802426} extended this study to the case of a Laguerre parameter $p$ proportional to $n$.

In the orthogonal and symplectic cases, the functional form of the  $O(n^{-1/3})$ terms was given by Forrester, Frankel, and Garoni \cite[Propositions~9--11]{MR2208159}.\footnote{However, the forms \cite[Eqs.~(5.1/4)]{MR2208159} of the $O(n^{-1/3})$ terms stated for the Gaussian cases are systematically affected by a typo introduced in \cite[Eq.~(3.14)]{MR2208159}, where $(2n)^{1/2}$ should be corrected to $(2n+1)^{1/2}$; see \cite[Remark 1.7]{MR4712577}. The same remark applies to \cite[Eqs.~(7.9), (7.101), (7.149)]{MR2641363}.} One can observe (see, e.g., \cite[Remark 1.7]{MR4712577}) that those $O(n^{-1/3})$ terms are actually the derivatives of the leading terms and can thus be eliminated by applying a shift in $n$ to $\mu_n$, revealing an $O(n^{-2/3})$ rate.

\subsection{Asymptotic expansions}

Continuing our work \cite{B25-1} on asymptotic expansions at the soft edge (which generalizes previous work \cite{MR2294977,MR3025686,MR2888709} on optimal convergence rates), we will prove in Section~\ref{sect:expansions} that applying a shift in $n$ to the scaling parameters $\mu_n, \sigma_n$ leads to `second-order accuracy', that is, to an optimal convergence rate of $O(n^{-2/3})$. In fact, we show that the rate estimate embeds into an asymptotic expansion in powers of a parameter $h_n$, where $h_n \asymp n^{-2/3}$. Specifically, we introduce the shifts (also applying to the parameter $p$)
\begin{subequations}\label{eq:nprime}
\begin{equation}\label{eq:nprimeDef}
n' = 
\begin{cases}
n - \frac12, &\quad \beta = 1,\\*[1mm]
n,           &\quad \beta = 2,\\*[1mm]
2n+ \frac12, &\quad \beta = 4,
\end{cases}
\end{equation}
noting that the left hand sides of \eqref{eq:rho14Intro} are then subject to exactly the same modification,
\begin{equation}\label{nprime14}
(2n+1)'\big|_{\beta=1} = n'\big|_{\beta=4}.
\end{equation}
\end{subequations}
In Section~\ref{sect:expansions}, we infer from the kernel expansions established in \cite[Lemma~2.1/3.1]{B25-1} an asymptotic expansion of the form
\begin{subequations}\label{eq:omega}
\begin{equation}
\rho_{2,n}(x)\frac{dx}{ds}\bigg|_{x=\mu_n + \sigma_n s} = \sum_{j=0}^m \omega_{2,j}(s;\tau_n) h_n^j + h_n^{m+1}\cdot O(e^{-2s})\quad (n\to\infty),
\end{equation}
with a parameter $0\leq \tau_n\leq 1$ satisfying $\tau_n=0$ in the Gaussian case and $\tau_n\to0^+$ as $p\to\infty$ in the Laguerre case.
With certain polynomial coefficients $\tilde p_j,\tilde q_j,\tilde r_j\in \Q[s,\tau]$, we have
\begin{equation}
\omega_{2,j} = \tilde p_j\Ai^2 + \tilde q_j \Ai'^2 + \tilde r_j \Ai\Ai'.
\end{equation}
Similarly, for $\beta=1,4$, we get asymptotic expansions of the form
\begin{equation}\label{eq:rho4doubleAsymp}
\begin{aligned}
\rho_{1,n}(x)\frac{dx}{ds}\bigg|_{x=\mu_{n'} + \sigma_{n'} s} &= \sum_{j=0}^m \omega_{1,j}(s;\tau_{n'}) h_{n'}^j + h_{n'}^{m+1}\cdot O(e^{-2s}),\\*[2mm]
2\rho_{4,n}(x)\frac{dx}{ds}\bigg|_{x=\mu_{n'} + \sigma_{n'} s} &= \sum_{j=0}^m \omega_{4,j}(s;\tau_{n'}) h_{n'}^j + h_{n'}^{m+1}\cdot O(e^{-2s}),
\end{aligned}\quad (n\to\infty),
\end{equation}
where, with a {\em $\beta$-independent} set of certain polynomial coefficients $p_j, q_j, r_j, u_j, v_j\in \Q[s,\tau]$,\footnote{Here, $\llbracket\cdot \rrbracket$ denotes the Iverson bracket: $\llbracket \mathcal P \rrbracket=1$ if the proposition $\mathcal P$ is true, and $\llbracket \mathcal P \rrbracket=0$ otherwise.} 
\begin{equation}
\omega_{\beta,j} =  p_j\Ai^2 +  q_j \Ai'^2 +  r_j \Ai\Ai' +u_j \Ai\AI_\nu +v_j \Ai'\AI_\nu\Big|_{\nu = \llbracket \beta=1\rrbracket}.
\end{equation}
\end{subequations}
We note that $\omega_{1,j}$ and $\omega_{4,j}$ only differ in the choice of the antiderivative, that is $\AI_1$ or $\AI_0$, which is explained by \eqref{eq:rho14Intro}, \eqref{nprime14}, and the form of the expansions \eqref{eq:expanHerPrim}, \eqref{eq:expanLagPrim} for $\Psi^\sharp_n$.

Fig.~\ref{fig:1} illustrates that the asymptotic expansions \eqref{eq:omega} notably improve upon the leading-order approximations \eqref{eq:leadingorder}, even for comparatively small dimensions (cf. also \cite[Fig.~1.b]{B25-3}).

\begin{figure}[tbp]
\includegraphics[width=0.38\textwidth]{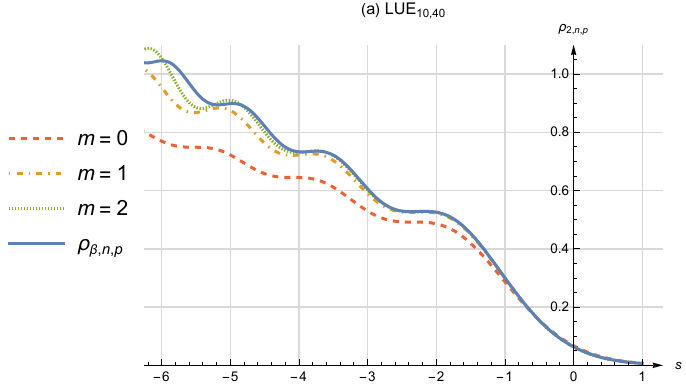}\hfil
\includegraphics[width=0.3\textwidth]{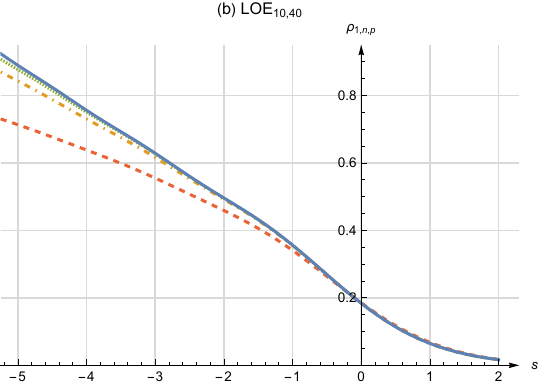}\hfil
\includegraphics[width=0.3\textwidth]{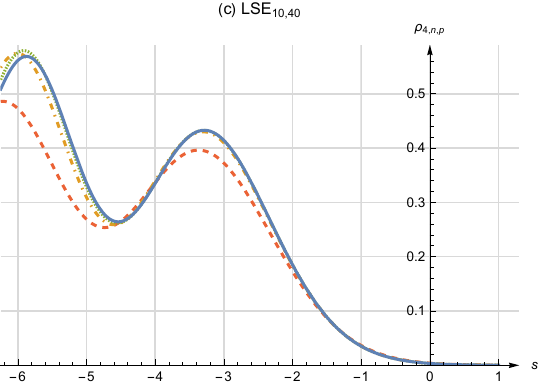}
\caption{{\footnotesize Plots of the rescaled densities $\rho_{\beta,n,p}$ for $n=10$, $p=40$ (blue solid line) together with the asymptotic expansions \eqref{eq:omega} truncated at $m=0,1,2$ (red dashed, orange dot-dashed, green dotted lines); the polynomial coefficients are taken from Theorems~\ref{thm:LUE}/\ref{thm:LOELSE}. Left panel (a): LUE ($\beta=2$); middle panel (b): LOE ($\beta=1$); right panel (c): LSE ($\beta=4$).}}
\label{fig:1}
\end{figure}

\subsection{Reconstruction of the expansion of the generating functions}\label{sect:GenTheory}
Finally, in Section~\ref{sect:generating}, we will relate the findings of this paper to the more general theory of \cite{B25-3}. There, we consider the generating function \cite[p.~329]{MR2641363}\footnote{For $0\leq \xi \leq 1$,
$E(x;\xi)$ has a simple probabilistic interpretation. If we keep each level of the underlying point process independently with probability $\xi$ and drop it otherwise, such a $\xi$-thinning yields
\[
\prob\big(\text{maximum level of a $\xi$-thinning is $\leq x$}\big) = E(x;\xi).
\]}
\[
E(x;\xi) = \int_\Omega\cdots \int_\Omega P(x_1,\ldots,x_n) \prod_{j=1}^n \big(1-\xi \llbracket x_j > x\rrbracket \big)\,dx_1\cdots dx_n,
\]
which encodes, among many other things, the distribution of the largest level for $\xi=1$ and the level density by expanding at $\xi=0$:
\begin{equation}\label{eq:derivedF}
\prob\big(\max_j x_j \leq x\big) = E(x;\xi)\Big|_{\xi=1},\qquad \rho(x) = \frac{\partial^2}{\partial x\partial\xi} E(x;\xi)\Big|_{\xi=0}.
\end{equation}
We prove in \cite{B25-3} for $\beta=2$, and conjecture for $\beta=1,4$, that
\begin{subequations}\label{eq:genTheory}
\begin{equation}
E_{\beta,n}(x;\xi)\Big|_{x=\mu_{n'} + \sigma_{n'} s}  = F_\beta(s;\xi) + \sum_{j=1}^m G_{\beta,j}(s,\tau_{n'};\xi)h_{n'}^j + h_{n'}^{m+1}O(e^{-s}),
\end{equation}
where, with certain polynomials $P_{\beta,j,k} \in \Q[s,\tau]$ that are independent of $\xi$,
\begin{equation}
G_{\beta,j} = \sum_{k=1}^{2j} P_{\beta,j,k}\cdot  \partial_s^k F_\beta.
\end{equation}
\end{subequations}
Moreover, there is a $\beta=1$ to $\beta=4$ duality that gives $P_{1,j,k} = P_{4,j,k}$. On one hand,
by the first relation in~\eqref{eq:derivedF}, the polynomials $P_{\beta,j,k}$ are exactly those established for the distribution of the largest level in \cite{B25-1}. On the other hand, by comparing coefficients with the expansions~\eqref{eq:omega} of the level densities proved in Section~\ref{sect:expansions}, we get  from the second equality in~\eqref{eq:derivedF} that
\begin{equation}\label{eq:omegaAlt}
\omega_{\beta,j} = \partial_s P_{\beta,j,1} \cdot \omega_{\beta,0} + \sum_{k=1}^{2j} (P_{\beta,j,k} + \partial_s P_{\beta,j,k+1})\cdot \partial_s^k \omega_{\beta,0} \quad (j=1,2,\ldots),
\end{equation}
setting $P_{\beta,j,2j+1}=0$.
Thus, using the Airy differential equation $\Ai''(s) = s \Ai(s)$ and the leading order terms \eqref{eq:leadingorder},
\[
\omega_{2,0}(s) = -s \Ai(s)^2 + \Ai'(s)^2, \quad \omega_{\beta,0}(s) = \omega_{2,0}(s) + \frac{1}{2}\Ai(s)\AI_{\llbracket \beta=1\rrbracket}(s) \quad(\beta=1,4),
\]
we get a structural consistency between \eqref{eq:omegaAlt} and \eqref{eq:omega}. The obvious way to use that relation would be to re-calculate the polynomials $\tilde p_j,\tilde q_j, \tilde r_j$ from the set $P_{2,j,k}$ and the polynomials $p_j, q_j,  r_j, u_j, v_j$ from the set $P_{1,j,k}$, which leads to a perfect agreement in all instances checked. However, as shown in Section~\ref{sect:generating}, using the algebraic independence, over the field $\C(s)$ of rational functions, of the three functions ($\nu$ fixed)
\[
\Ai, \Ai', \AI_\nu,
\]
recently established in \cite{B25-2}, it is possible to go the other way and uniquely reconstruct the first $14$ polynomials
\begin{equation}\label{eq:GenPoly}
\begin{gathered}
P_{\beta,1,1}, \quad P_{\beta,1,2} \quad (\beta=1,2,4),\\*[1mm] 
P_{\beta,2,1}, \quad P_{\beta,2,2}, \quad P_{\beta,2,3}, \quad P_{\beta,2,4} \quad (\beta=1,4),
\end{gathered}
\end{equation}
from $\tilde p_1,\tilde q_1, \tilde r_1$ and  $p_j, q_j,  r_j, u_j, v_j$ ($j=1,2$). The agreement with those polynomials calculated from the more general theory provides further compelling evidence supporting the conjectures underlying the cases $\beta=1,4$ in \cite{B25-3}, \cite{B25-1}: only the final form \eqref{eq:genTheory} of the result concerning the generating function is assumed here, without relying on any specific details of the theory.

\section{Matrix Ensembles and Skew-Orthonormal Wave Functions}\label{sect:ensembles}

This section recalls, first, some notation and facts from random matrix theory (cf., e.g., the monographs \cite{MR2760897,MR2641363,Mehta04}) and aims, second, at refining them by a unified construction of the skew-orthogonal polynomials associated with the classical Hermite and Laguerre weights. Based on working with the corresponding wave functions, this construction is considerably simpler than those in \cite{MR1762659, MR2560290}.

\subsection{Parametrization of the matrix ensembles}\label{subsect:ensembles}
We follow the choice of the parameters made in \cite{B25-1}. Functions related to the matrix ensembles are defined on $\Omega$, which is $(-\infty,\infty)$ in the Gaussian cases and $(0,\infty)$ in the Laguerre ones.
The joint probability density of the (unordered) levels $x_1,\ldots,x_n$ of an $n$-dimensional ensemble  is given in the form
\begin{equation}\label{eq:joint}
P_\beta(x_1,\ldots,x_n) \propto |\Delta(x_1,\ldots,x_n)|^\beta \prod_{j=1}^n w_\beta(x_j),
\end{equation}
where $\Delta(x_1,\ldots,x_n) = \prod_{j>k} (x_j - x_k)$ denotes the Vandermonde determinant and the weight functions $w_\beta(\xi)$ are given, for the Gaussian ensembles, by the Hermite weights
\begin{subequations}\label{eq:alphabeta}
\begin{equation}
w_\beta(x) = e^{-c_\beta x^2} \quad (x\in\R)
\end{equation}
and, for the Laguerre ensembles with $\alpha>-1$, by the Laguerre weights
\begin{equation}
w_\beta(x) = x^\alpha e^{-c_\beta x} \quad (x>0).
\end{equation}
Here, $\beta=1,2,4$ encodes whether the ensemble is orthogonal, unitary, or symplectic. However, we replace the Laguerre parameter $\alpha$ with the parameter (generally a real number)
\begin{equation}
p= n - 1 + \frac{2 (\alpha + 1)}{\beta} > n-1.
\end{equation}
If $p$ is an integer, the levels of the Laguerre ensemble are distributed as the eigenvalues of the $n$-variate Wishart distribution with $p$ degrees of freedom. Since the Wishart distribution exhibits a $p\leftrightarrow n$ symmetry \cite[§1]{B25-1}, the formulas for the  Laguerre ensemble can be put to a form reflecting that symmetry, which results in considerable simplifications. Consequently, we write $\GbE_n$ for the Gaussian ensembles and $\LbE_{n,p}$ for the Laguerre ensembles. We fix the scale of the levels by
\begin{equation}
c_\beta = \begin{cases}
1/2 &\quad \beta = 1,\\*[1mm]
1 & \quad \beta = 2,4,
\end{cases}
\end{equation}
\end{subequations}
which facilitates the Forrester--Rains interrelations \cite[§§4--5]{MR1842786}\cite[Theorem~2]{MR2461989},
\begin{equation}\label{eq:ForresterRains}
\
\GSE_n = \even(\GOE_{2n+1}),\quad \LSE_{n,p} = \even(\LOE_{2n+1,2p+1}),
\end{equation}
meaning that the ordered levels of the symplectic cases on the left are distributed in the same way as every second level of the corresponding orthogonal cases on the right. 

\subsection{Index consistency rule}\label{sect:index}
In the interrelation \eqref{eq:ForresterRains}, and many formulae to follow for the Laguerre case, we get double indices $j,\hat\jmath$ that take the parallel form 
\[
j= \mu n + \kappa, \quad \hat\jmath = \mu p + \kappa,
\]
which quickly becomes cumbersome and repetitive. Fixing $p$ for a given $n$, the second index carries no significant additional information and is omitted as far as possible. It can be reconstructed according to the following consistency rule: every formula carries a {\em unique} hidden parameter $q>-1$, such that single indices $j$ represent index pairs $j,\hat\jmath$ with
\[
\hat\jmath = 
\begin{cases} 
j + q &\text{for terms defined at $\beta=1,2$ as well as for the wave functions},\\*[1mm]
j + q/2 &\text{for terms defined at $\beta=4$}.
\end{cases} 
\]
For instance, the hidden parameter in the Laguerre part of \eqref{eq:ForresterRains} is consistently $q=2(p-n)$ when the $\beta=4$ rule is applied to the left-hand side and the $\beta=1$ rule to the right.

\begin{remark} The special treatment of $\beta=4$ can be avoided by defining the dimension of a symplectic ensemble $\SE_n$ as $\bar n = 2n$ (adjusting, likewise, the Laguerre parameter to $\bar p = 2p$). Each level then acquires multiplicity two, as when quaternions in the matrix model are represented by complex~$2\times 2$ matrices. Hence, $2\rho_{4,n}$ is replaced by $\rho_{4,\bar n}$ in \eqref{eq:rho4double}, \eqref{eq:rho4doubleAsymp}, and the $\beta=4$ case of the transformation \eqref{eq:nprimeDef} reduces to $\bar n' = \bar n + \tfrac{1}{2}$, making the duality between the orthogonal and symplectic cases even more transparent. We refrain, however, from adopting this convention, as it departs from standard usage.
\end{remark}

According to \eqref{eq:alphabeta}, here and in what follows, the Laguerre weight functions associated with a hidden parameter $q$ are given by\footnote{Relating $w_1$ to $\sqrt{w_2}$ and $\sqrt{w_4}$ arises from the fact that the skew-inner product $\langle\cdot,\cdot\rangle_1$ is defined as a double integral, whereas the (skew-)inner products $\langle\cdot,\cdot\rangle_2$ and $\langle\cdot,\cdot\rangle_4$ are defined as single integrals; see \eqref{eq:inner}.}
\begin{subequations}\label{eq:qweights}
\begin{equation}
w_1(x) = x^{(q-1)/2} e^{-x/2}, \quad \sqrt{w_2(x)} = x^{q/2} e^{-x/2},\quad \sqrt{w_4(x)} = x^{(q+1)/2} e^{-x/2}.
\end{equation}
In particular, we have $\sqrt{w_4} = w_2/w_1$, which is exactly the condition in \cite[Theorem 4.7]{MR1842786} underlying \eqref{eq:ForresterRains}. This condition is also satisfied in the case of the Gaussian ensembles, where
\begin{equation}
w_1(x) = \sqrt{w_2(x)} = \sqrt{w_4(x)} = e^{-x^2/2}.
\end{equation}
\end{subequations}
To combine similar formulae for $\GbE$ and $\LbE$ in a single notation, we associate the Gaussian ensembles formally with an omittable second index $\infty$, corresponding to a hidden parameter $q=\infty$. This notation is suggested by the $\LbE_{n,p}\to \GbE_n$ transition law for $p\to\infty$.\footnote{See, e.g., \cite[Appendix B]{B25-1}, \cite{MR3413957} or, in the case of the $\LOE$ with integer $p$, \cite[Corollary 13.3.2]{MR1990662} and \cite[Corollary 9.5.7]{MR652932}.} In fact, for $n$ fixed, all formulae for the Gaussian ensembles can be obtained from those for the Laguerre ensembles by using the scaling limits \eqref{eq:scalingLaw} discussed in Appendix~\ref{app:ScalingLaw}. However, we will give separate proofs for the Gaussian ensembles and Hermite wave functions, serving as a consistency check. 

\subsection{Skew inner products} We consider the class ${\mathcal F}$ of integrable smooth functions $f$ on $\Omega$ with exponential decay as $x\to \pm \infty$.
We define the integral operator $\epsilon$ by
\begin{equation}\label{eq:epsilon}
(\epsilon f)(x) = \frac{1}{2}\int_\Omega \sgn(x-y) f(y)\,dy = -\int_x^\infty f(y)\,dy + \frac{1}{2}\int_\Omega f(y)\,dy \quad (f \in {\mathcal F}).
\end{equation}
On one hand, $\epsilon f$ is an antiderivative of $f$, satisfying $(\epsilon f)'=f$, with the constant of integration fixed by
\begin{equation}\label{eq:int}
\int_\Omega f(y)\,dy = \epsilon f^\sharp(\infty), \quad f^\sharp = 2f.
\end{equation}
On the other hand, when $f'\in{\mathcal F}$, we have that $\epsilon(f')=f$ is unconditionally true in the Gaussian case, and in the Laguerre case if $f(0)=0$. Now, when the corresponding integrands belong to ${\mathcal F}$, we define the $L^2$ inner product and the skew inner products for $\beta=1,4$ by\footnote{At various places in the literature, the skew inner products differ in sign or come without the factor $1/2$.}
\begin{equation}\label{eq:inner}
\langle f,g\rangle_2 = \int_\Omega f(x)g(x)\,dx, \quad 
\langle f,g\rangle_1 = \langle  f,\epsilon g\rangle_2, \quad
\langle f,g\rangle_4 = \frac{1}{2}\left(\langle  f',g\rangle_2 - \langle  f,g'\rangle_2\right).
\end{equation}
By a partial integration, we see that the simplification $\langle f,g\rangle_4 = \langle  f',g\rangle_2$ 
is unconditionally true in the Gaussian case, and in the Laguerre case if $f(0)=0$ or $g(0)=0$.

\subsection{Orthonormal wave functions}\label{app:func} The Hermite and Laguerre wave functions\footnote{In lieu of a better name, we call them `wave functions' since those associated with the Hermite polynomials are the oscillator wave functions of quantum physics.} are 
\begin{equation}\label{eq:phi}
\phi_{n,\infty}(x) = \frac{e^{-x^2/2}}{\pi^{1/4}\sqrt{2^nn!}} H_n(x), \quad \phi_{n,p}(x) = \sqrt{\frac{n!}{\Gamma(p+1)}}\, x^{(p-n)/2} e^{-x/2}(-1)^n L_n^{(p-n)}(x),
\end{equation}
where $H_n$ and $L_n^{(\alpha)}$ denote the Hermite and Laguerre polynomials as defined in \cite[§§6.1--2]{MR1688958}.
Since the definition of the Laguerre polynomials extends to arbitrary complex values of the parameter $\alpha$, the associated wave functions are well-defined for any real $p>-1$. If $p$ is an integer, the fundamental $p\leftrightarrow n$ symmetry of the Laguerre wave functions mentioned in Section~\ref{subsect:ensembles} is revealed by \cite[Eq.~(5.2.1)]{MR0372517}, which can be put to the form
\begin{equation}\label{eq:Whittaker}
n!\, x^{(p-n)/2} (-1)^n L_n^{(p-n)}(x) = p\hspace{0.2mm}!\, x^{(n-p)/2} (-1)^p L_p^{(n-p)}(x)\quad (n,p=0,1,2,\ldots).
\end{equation}
Using the index consistency rule of Section~\ref{sect:index}, we can rewrite the known orthogonality relations   \cite[Eqs.~(6.1.5/6.2.3)]{MR1688958} of the Hermite and Laguerre polynomials in form of the orthonormality of the associated wave functions,
\begin{equation}\label{eq:L2ortho}
\langle \phi_n,\phi_{m}\rangle_2 = \llbracket n=m\rrbracket.
\end{equation}
Here, for reasons of integrability in the Laguerre case, we must impose the constraint $q > - 1$ on the hidden real parameter $q$.
By construction, the functions $\phi_n/\sqrt{w_2}$ are the classical orthonormal polynomials of degree~$n$ associated with the weight $w_2$ displayed in \eqref{eq:qweights}.

\begin{remark} Rescaling the three-term-recurrences \cite[Eqs.~(6.1.10/6.2.5)]{MR1688958} of the Hermite and Laguerre polynomials yields, for $n=0,1,2,\ldots$\,, that
\[
\begin{aligned}
x \phi_{n,\infty} &= \sqrt{\frac{n}2}\, \phi_{n-1,\infty} + \sqrt{\frac{n+1}2}\, \phi_{n+1,\infty},\\*[2mm]
(x-(n+p+1)) \phi_{n,p} &= \sqrt{\phantom{(}\!\!\!np} \,\phi_{n-1,p-1} + \sqrt{(n+1)(p+1)} \,\phi_{n+1,p+1},\\*[2mm]
\end{aligned}
\]
initialized in both cases with $\phi_{-1}=0$ and $\phi_0 \propto \sqrt{w_2}$, so that $\langle \phi_0,\phi_0\rangle_2 = 1$. However, we will not use these recurrences; their role will be taken by those in Theorem~\ref{thm:diffchi}.
\end{remark}

\subsection{Bi-orthonormal wave functions}\label{app:funcSkew}

We consider the induced function systems
\begin{equation}\label{eq:chipsi}
\begin{aligned}
\chi_{n,\infty}(x) &= \gamma_n \phi_{n,\infty}(x),& \psi_{n,\infty}(x) &= \gamma_n^{-1} \phi_{n,\infty}(x),\\*[2mm]
\chi_{n,p}(x) &= \gamma_n\gamma_p(x/2)^{1/2} \phi_{n,p}(x),& \psi_{n,p}(x) &= \gamma_n^{-1}\gamma_p^{-1}(x/2)^{-1/2} \phi_{n,p}(x) \quad (p<\infty),
\end{aligned}
\end{equation} 
where the normalization coefficients\footnote{Choosing them in this particular way ensures that the formulae displayed in Theorems~\ref{thm:diffchi}/\ref{thm:E}, Corollary~\ref{cor:int}, and Eq.~\eqref{eq:pGOE} become especially simple all at once.} (the second variant is obtained by applying Legendre's duplication formula \cite[Eq.~(1.5.1)]{MR1688958})
\begin{equation}\label{eq:gammaNu}
\gamma_\nu = \frac{2^{(\nu+1)/2} \Gamma\left(\frac{\nu+1}{2}\right)}{\pi^{1/4}\sqrt{\Gamma(\nu+1)}} = \sqrt{\frac{2\,\Gamma\!\left(\frac{\nu+1}{2}\right)}{\Gamma\!\left(\frac{\nu}{2}+1\right)}} \quad (\nu > -1),
\end{equation} 
satisfy the functional equation
\begin{equation}\label{eq:agamma}
\frac12\gamma_\nu\gamma_{\nu-1} = \sqrt\frac2{\nu} \quad (\nu>0).
\end{equation}
Using the index consistency rule of Section~\ref{sect:index}, we write $\chi_n$ and $\psi_n$ if convenient. To absorb frequent factors of $2$ (which cannot be taken care of otherwise), we follow the convention introduced in \eqref{eq:int} and define
\begin{equation}\label{eq:psiSharp}
\psi^\sharp_n = 2 \psi_n.
\end{equation}
From \eqref{eq:L2ortho} we get the $L^2$-bi-orthonormality
\begin{equation}\label{eq:biortho}
\langle \chi_{n},\psi_{m}\rangle_2 = \llbracket n=m\rrbracket.
\end{equation}
Now, the bi-orthogonal wave functions enjoy a very simple `structure relation' (cf.~\cite{MR2345243}).
\begin{theorem}\label{thm:diffchi} With the definition $\psi_{-1}=0$, the bi-orthonormal system of wave functions satisfies, for $n=0,1,2,\ldots$, the differential and integral recursions
\begin{equation}\label{eq:chiRec}
\chi_{n}' = \psi^\sharp_{n-1} - \psi^\sharp_{n+1}, \quad \chi_{n} = \epsilon\psi^\sharp_{n-1} - \epsilon\psi^\sharp_{n+1}.
\end{equation}
\end{theorem}
\begin{proof} The second recursion follows from the first one by applying the integral operator $\epsilon$ and observing 
\begin{equation}\label{eq:epsilonchiprime}
\epsilon(\chi_n') = \chi_n.
\end{equation}
In the Hermite case, this is unconditionally true. It follows in the Laguerre case as follows: by spelling out the hidden parameter $q>-1$, we have
\[
\chi_n(x) \propto x^{(q+1)/2} e^{-x/2} L_n^{(q)}(x)
\]
with $(q+1)/2>0$, so that the condition $\chi_n(0)=0$ for \eqref{eq:epsilonchiprime} to be true is satisfied.

In the Hermite case, from the differentiation formula \cite[Eq.~(6.1.11)]{MR1688958} and the three-term recurrence \cite[Eq.~(6.1.10)]{MR1688958} we infer the structure relation
\[
\frac{d}{dx}\left(e^{-x^2/2} H_n(x)\right) = \frac{1}{2} e^{-x^2/2}\left(2n  H_{n-1}(x) - H_{n+1}(x)\right),
\]
which implies the assertion by a rescaling using \eqref{eq:agamma}. 

Similarly, in the Laguerre case, we infer from the differentiation formula \cite[Eq.~(6.2.6)]{MR1688958} and the three-term recurrence \cite[Eq.~(6.2.5)]{MR1688958} the structure relation
\[
\frac{d}{dx}\left(x^{(q+1)/2} e^{-x/2}L_n^{(q)}(x)\right) = - \frac{1}{2}x^{(q-1)/2} e^{-x/2}\left((n+q) L_{n-1}^{(q)}(x) - (n+1)L_{n+1}^{(q)}(x)\right),
\]
which implies the assertion by a further rescaling using \eqref{eq:agamma}.
\end{proof}

\begin{corollary}\label{cor:int} The bi-orthonormal wave functions $\psi_n$ integrate to
\[
\int_\Omega \psi_{n}(x)\,dx = \llbracket\text{\rm$n$ is even}\rrbracket.
\]
\end{corollary}

\begin{proof} By \eqref{eq:int}, evaluating the integral recursion in \eqref{eq:chiRec} at $x=\infty$ gives
\[
\int_\Omega \psi_{n}(x)\,dx = \int_\Omega\psi_{n-2}(x)\,dx.
\]
Because of $\psi_{-1}=0$ this evaluates recursively to $0$ if $n$ is odd. If $n$ is even, this evaluates recursively to the integral of $\psi_0$, which is, in the Hermite case 
\[
\int_\Omega \psi_{0}(x)\,dx = \frac{1}{\sqrt{2\pi}} \int_{-\infty}^\infty e^{-x^2/2} \,dx = 1,
\]
and, in the Laguerre case with hidden parameter $q > -1$, 
\[
\int_\Omega \psi_{0}(x)\,dx = \frac{1}{\Gamma((q+1)/2)}\int_0^{\infty} e^{-t} t^{(q+1)/2-1} \,dt = 1,
\]
where the Euler integral is obtained by substituting $x=2t$.
\end{proof}

\begin{remark} Alternatively, the integrals in Corollary~\ref{cor:int} can be established head-on. Using symmetry, the odd case of the Hermite wave functions is clear, and the even case follows from 
\[
\int_{-\infty}^\infty e^{-x^2/2} H_n(x)\,dx = \int_0^\infty e^{-x/2} x^{-1/2} H_n\big(x^{1/2}\big)\,dx = \sqrt{2\pi} \frac{n!}{(n/2)!}\qquad\text{($n$ even)},
\]
which is a Laplace transform \cite[Eq.~(3.25.1.9)]{MR1162979} evaluated at $1/2$.
Similarly, the Laguerre case follows from the integral (with $\alpha>-1$)
\[
\int_0^\infty e^{-x/2} x^{(\alpha-1)/2} L^{(\alpha)}_n(x)\,dx =
\begin{cases}
\dfrac{2^{\frac{\alpha +1}{2}} \Gamma\big(\frac{1}{2} (n+\alpha +1)\big)}{\Gamma\big(\frac{n}{2}+1\big)}  & \text{($n$ even)},\\*[2mm]
0 & \text{($n$ odd)},
\end{cases}
\]
which is also a Laplace transform \cite[Eq.~(3.24.1.10)]{MR1162979} evaluated at $1/2$. Yet another derivation of the last formula, employing the hypergeometric function $_2F_1$, can be found in \cite[Eq.~(A.2)]{MR1313625}.
\end{remark}

\subsection{Skew-orthonormal wave functions}
By construction, keeping the index consistency rule and \eqref{eq:qweights} in mind,
\[
\psi_n/w_1, \quad \chi_n/\sqrt{w_4},
\]
are polynomials of degree $n$. Theorem~\ref{thm:diffchi} shows recursively that also 
\[
\chi'_{n}/w_1, \quad (\epsilon \psi_{2n+1})/\sqrt{w_4}, 
\]
are polynomials of degree $n+1$ and $2n$. The following corollary of Theorem~\ref{thm:diffchi} shows that mixing them according to parity will result in the skew-orthonormal polynomials for $\beta=1,4$.

\begin{corollary}\label{cor:skew} Define the function system $U_0^{(1)},U_1^{(1)},U_2^{(1)},\ldots$\, by the sequence
\[
\psi_0,\chi_0',\psi_2,\chi_2',\psi_4,\chi_4',\ldots\,,
\]
and the function system $U_0^{(4)},U_1^{(4)},U_2^{(4)},\ldots$\, by the sequence
\[
\epsilon\psi_1,\chi_1,\epsilon\psi_3,\chi_3,\epsilon\psi_5,\chi_5,\ldots\,.
\]
The systems are skew-orthonormal for $\beta=1,4$,
\[
\Big(\big\langle U_j^{(\beta)}, U_k^{(\beta)}\big\rangle_\beta\Big)_{j,k=0}^{2n-1} = I_n \otimes 
\begin{pmatrix}
0 & 1\\
-1 & 0
\end{pmatrix}.
\]
Hence, $U_n^{(1)}/w_1$ and $U_n^{(4)}/\sqrt{w_4}$ are the skew-orthonormal polynomials of degree~$n$ for $\beta=1,4$.
\end{corollary}
\begin{proof}
By the skew-symmetry, we have to show that, for $n,m=0,1,2,\ldots$\,, 
\[
\big\langle U_{2n}^{(\beta)}, U_{2m}^{(\beta)}\big\rangle_\beta = \big\langle U_{2n+1}^{(\beta)}, U_{2m+1}^{(\beta)}\big\rangle_\beta = 0,\quad \big\langle U_{2n}^{(\beta)}, U_{2m+1}^{(\beta)}\big\rangle_\beta =\llbracket n=m \rrbracket.
\]
Because of \eqref{eq:inner} and \eqref{eq:epsilonchiprime}, this amounts to showing that
\[
\langle \psi_n, \epsilon \psi_m \rangle_2 = 0, \quad \langle \chi_n', \chi_m \rangle_2 =0, \quad \langle \psi_n, \chi_m \rangle_2 = \llbracket n=m \rrbracket \quad (\text{$n$, $m$ same parity}).
\]
The last equality follows from the bi-orthogonality~\eqref{eq:biortho}, the second one from Theorem~\ref{thm:diffchi} and the bi-orthogonality giving
\[
\langle \chi_n', \chi_m \rangle_2 = 2\langle \psi_{n+1}, \chi_m \rangle_2 -2\langle \psi_{n-1}, \chi_m \rangle_2= 0 \quad (\text{$n$, $m$ same parity}).
\]
For the first one, if $n,m$ are odd, Theorem~\ref{thm:diffchi} and the bi-orthogonality give recursively that 
\[
\langle \psi_n, \epsilon \psi_m \rangle_2 = \langle \psi_n, \epsilon \psi_{m-2} \rangle_2 = \langle \psi_n, \epsilon \psi_{-1} \rangle_2 = 0,
\]
where we have used $\psi_{-1}=0$. If $n,m$ are even, we get similarly
\[
\langle \psi_n, \epsilon \psi_m \rangle_2 = \langle \psi_n, \epsilon \psi_{m-2} \rangle_2 = \langle \psi_n, \epsilon \psi_0 \rangle_2 = -\langle \psi_0, \epsilon \psi_n \rangle_2 = -\langle \psi_0, \epsilon \psi_0 \rangle_2 = 0,
\]
where we have used the skew-symmetry of the operator $\epsilon$.
\end{proof}

\begin{remark}
We calculate with the technique used in the proof that
\begin{equation}\label{eq:plusminusone}
\langle \psi_n, \epsilon\psi^\sharp_{n+1}\rangle_2 = -\llbracket\text{$n$ is even}\rrbracket, \quad \langle \psi_n, \epsilon\psi^\sharp_{n-1}\rangle_2 = \llbracket\text{$n$ is odd}\rrbracket.
\end{equation}
Together with Corollary~\ref{cor:int}, this can be used to check \eqref{eq:rho4} and \eqref{eq:rho1oddeven} against \eqref{eq:mass}.
\end{remark}

\section{Level Densities Written in Terms of Wave Functions}\label{sect:densities}

The basic formulae for the level densities are given at several places; e.g., \cite{MR1762659,MR2641363,MR1190440,Mehta04,MR1142971,MR1657844,MR1675356}. However, revisiting the formulae in terms of the $L^2$-bi-orthonormal system of wave functions introduced in Section~\ref{app:funcSkew} yields significant simplifications. We also show that there is no need to distinguish between the cases of even and odd dimension for the orthogonal ensembles; in fact, the expression of the level density can be brought to a form that holds independently of the parity of $n$. 

For the bi-orthonormal wave functions $\psi_n$ defined in \eqref{eq:chipsi}, we consider the particular antiderivative $\Psi_{n}$ with constant of integration fixed by taking a value $0$ at infinity, that is,
\begin{equation}\label{eq:PsiIntro}
\Psi^\sharp_{n}(x) = - \int_x^\infty \psi^\sharp_{n}(t)\,dt = \epsilon\psi^\sharp_{n}(x) - \llbracket\text{$n$ is even}\rrbracket.
\end{equation}
Here, the second form follows from \eqref{eq:epsilon} and Corollary~\ref{cor:int}.

\subsection{Unitary ensembles} 
As for any determinantal point process (see, e.g., \cite[§4.2]{MR2760897}), the level densities of the unitary ensembles are given as the diagonal of the corresponding correlation kernels. So, using the notation as in \cite[§§2.1,3.1]{B25-1} and Section~\ref{sect:ensembles}, we get
\begin{subequations}
\begin{equation}\label{eq:rho2Kernel}
\rho_{2,n,\infty}(x) = K_n^{\GUE}(x,x) = \sum_{j=0}^{n-1} \phi_{j,\infty}(x)^2,\quad
\rho_{2,n,p}(x) = K_{n,p}^{\LUE}(x,x) = \sum_{j=0}^{n-1} \phi_{j,\hat{\jmath}}(x)^2,
\end{equation}
where $\hat{\jmath}$ satisfies $\hat{\jmath}-j = p-n$. Using the index consistency rule of Section~\ref{sect:index}, we combine both cases into the single expression
\begin{equation}\label{eq:rho2}
\rho_{2,n}(x) = \sum_{j=0}^{n-1} \phi_j(x)^2 = \sum_{j=0}^{n-1} \chi_{j}(x)\psi_{j}(x).
\end{equation}
\end{subequations}

\subsection{Symplectic ensembles}\label{sect:rho4} Using the representation of the $\beta=4$ skew-orthonormal polynomials in Corollary~\ref{cor:skew}, we immediately infer from the work of Tracy and Widom \cite[Eq.~(3.4) and p.~829]{MR1657844} that
\[
2\rho_{4,n} = \sum_{j=0}^{n-1} 
\begin{pmatrix}
U_{2j}^{(4)\prime} & U_{2j+1}^{(4)\prime} 
\end{pmatrix}
\begin{pmatrix}
0 & 1\\*[2mm]
-1 & 0
\end{pmatrix}
\begin{pmatrix}
U_{2j}^{(4)}\\*[2mm]
U_{2j+1}^{(4)} 
\end{pmatrix} = \sum_{j=0}^{n-1} \chi_{2j+1}\psi_{2j+1} -  \sum_{j=0}^{n-1} \chi_{2j+1}' \epsilon\psi_{2j+1}.
\]
If $(a_{ij})$ is the skew-symmetric tridiagonal matrix with $a_{j,j-1}=2$, then by Theorem~\ref{thm:diffchi}
\[
\chi_j' = \sum_{k\geq0} a_{jk}\psi_k, \quad \chi_j = \sum_{k\geq0} a_{jk} \epsilon\psi_k.
\]
Hence, following the argument on \cite[p.~832]{MR1657844}, we obtain
\begin{align*}
\sum_{j=0}^{n-1} \chi_{2j+1}' \epsilon\psi_{2j+1} &= \sum_{\stackrel{j,k\geq 0}{\phantom{\big|}j \text{ odd} \leq 2n-1}} a_{jk}\psi_k \cdot\epsilon\psi_j = - 2\psi_{2n} \epsilon\psi_{2n+1} - \sum_{\stackrel{j,k\geq 0}{\phantom{\big|}k \text{ even} \leq 2n}}  \psi_k \cdot a_{kj} \epsilon\psi_j  \\
&= - 2\psi_{2n} \epsilon\psi_{2n+1} -\sum_{j=0}^n \chi_{2j}\psi_{2j},
\end{align*}
so that we have, by using \eqref{eq:PsiIntro}, \eqref{eq:rho2} and Theorem~\ref{thm:diffchi} once more, the concise expressions (in the Gaussian case, taking \eqref{eq:agamma} into account, the first form can be found in \cite[p.~833]{MR1657844})
\begin{equation}\label{eq:rho4}
\begin{aligned}
2\rho_{4,n} &= \rho_{2,2n+1} +  \psi_{2n} \cdot\epsilon \psi^\sharp_{2n+1} = \rho_{2,2n+1} + \psi_{2n} \cdot\Psi^\sharp_{2n+1}\\*[2mm]
&=\rho_{2,2n} + \psi_{2n} \cdot\epsilon \psi^\sharp_{2n-1}  = \rho_{2,2n} + \psi_{2n} \cdot \Psi^\sharp_{2n-1}.
\end{aligned}
\end{equation}
Alternatively, the latter form follows from rewriting the formulae \cite[Eq.~(5.3/5.11)]{MR2208159} (which spell out the details in \cite[Proposition 4.5]{MR1762659}) in terms of the Hermite and Laguerre wave functions and canceling constants.

\subsection{Orthogonal ensembles} We start with an even dimension.
Using the representation of the $\beta=1$ skew-orthonormal polynomials in Corollary~\ref{cor:skew}, we immediately infer from the work of Tracy and Widom \cite[Eq.~(3.4) and p.~831]{MR1657844} that
\[
\rho_{1,2n} = \sum_{j=0}^{2n-1} 
\begin{pmatrix}
U_{2j}^{(1)} & U_{2j+1}^{(1)} 
\end{pmatrix}
\begin{pmatrix}
0 & 1\\*[2mm]
-1 & 0
\end{pmatrix}
\begin{pmatrix}
\epsilon U_{2j}^{(1)}\\*[2mm]
\epsilon U_{2j+1}^{(1)} 
\end{pmatrix} = \sum_{j=0}^{n-1} \chi_{2j}\psi_{2j} -  \sum_{j=0}^{n-1} \chi_{2j}' \epsilon\psi_{2j}.
\]
With the same calculation as in the symplectic case, we get
\[
\sum_{j=0}^{n-1} \chi_{2j}' \epsilon\psi_{2j} = - 2\psi_{2n-1} \epsilon\psi_{2n} -\sum_{j=0}^{n-1} \chi_{2j+1}\psi_{2j+1},
\]
so that we have, by using \eqref{eq:rho2} and Theorem~\ref{thm:diffchi}, the concise expressions (in the Gaussian case, taking \eqref{eq:agamma} into account, the first form can be found in \cite[p.~834]{MR1657844})
\begin{subequations}\label{eq:rho1oddeven}
\begin{equation}\label{eq:rho1even}
\rho_{1,n} = \rho_{2,n} + \psi_{n-1}\cdot \epsilon\psi^\sharp_{n} = \rho_{2,n-1} + \psi_{n-1}\cdot \epsilon\psi^\sharp_{n-2}\quad \text{($n$ even)}.
\end{equation}
Alternatively, the latter form follows from rewriting the formulae \cite[Eq.~(4.1/5.10)]{MR2208159} (which spell out the details in \cite[Proposition 4.2]{MR1762659}) in terms of the Hermite and Laguerre wave functions and canceling constants.

For an odd dimension, one possibility to derive the corresponding formulae is to spell out the details of \cite[Proposition 4.3]{MR1762659}\footnote{Note that there is a weight factor missing, which is corrected in \cite[Eq.~(6.112)]{MR2641363}.} 
or \cite[§5.7]{Mehta04}. Instead, we proceed by using Corollary~\ref{cor:rho1}, which infers the relation
\[
\rho_{1,2n+1} = 2\rho_{4,n} + \psi_{2n}
\]
as a probabilistic consequence of \eqref{eq:ForresterRains}. Combining this with \eqref{eq:rho4}
yields
\begin{equation}\label{eq:rho1odd}
\rho_{1,n} = \rho_{2,n}+ \psi_{n-1} \cdot \big(1 + \epsilon\psi^\sharp_{n}\big) = \rho_{2,n-1}+ \psi_{n-1} \cdot \big(1 + \epsilon\psi^\sharp_{n-2}\big)  \quad \text{($n$ odd)}.
\end{equation}
\end{subequations}
Even though the formulae for the even cases~\eqref{eq:rho1even} and the odd ones \eqref{eq:rho1odd} look superficially different, they can easily be brought to the same form: by paying attention to \eqref{eq:PsiIntro}, we obtain, for each parity of $n$, 
\begin{equation}\label{eq:rho1}
\rho_{1,n}  = 
\rho_{2,n} + \psi_{n-1}\cdot \big(1  + \Psi^\sharp_{n}\big) = \rho_{2,n-1} + \psi_{n-1}\cdot \big(1  + \Psi^\sharp_{n-2}\big).
\end{equation}

\section{Asymptotic Expansions of the Level Densities}\label{sect:expansions}

\subsection{Gaussian Ensembles}

Here, following \cite{B25-1}, the scaling and expansion parameters are 
\begin{equation}\label{eq:GUEscaling}
\mu_{n,\infty}=\sqrt{2n}, \quad \sigma_{n,\infty} = 2^{-1/2}n^{-1/6}, \quad h_{n,\infty} =\frac{\sigma_n}{2\mu_n}= 4^{-1} n^{-2/3}.
\end{equation}
For $\beta=1,2,4$, we observe $h_{n',\infty} \asymp n^{-2/3}$. As shown in Appendix~\ref{app:ScalingLaw}, these quantities can be understood as a scaling limit $p\to\infty$ of the corresponding quantities defined in \eqref{eq:LUEscaling}.

Using the expansion \cite[Lemma~2.1]{B25-1} of the correlation kernel at the soft edge, we immediately infer from \eqref{eq:rho2Kernel} the following theorem for the Gaussian unitary ensemble.\footnote{The case $m=1$ was previously established with an error of $O(n^{-1})$ in \cite[Eq.~(72)]{MR2178598} and \cite[Eq.~(3.32)]{MR2233711}.}

\begin{theorem}\label{thm:GUE}
There holds, with $m$ being any fixed integer in the range $0\leq m\leq m_*$,\footnote{\label{fn:mstar}For the reasons explained in \cite[Remark~1.1]{B25-1}, we have to impose a bound $m_*$ up to which certain properties have explicitly been verified (we did so up to $m_*=10$); we expect this bound to be insignificant since the Riemann--Hilbert asymptotic analysis of \cite{arXiv:2309.06733} most likely extends to the realm of  \cite[Lemma~2.1]{B25-1}.}
\begin{subequations}\label{eq:GUE}
\begin{equation}
\rho_{2,n,\infty}(x)\frac{dx}{ds}\bigg|_{x=\mu_{n,\infty} + \sigma_{n,\infty} s} = \sum_{j=0}^m \omega_{2,j}(s) h_{n,\infty}^j + h_{n,\infty}^{m+1}\cdot O(e^{-2s}) \quad (n\to\infty),
\end{equation}
uniformly for $s$ bounded from below. Here, the expansion terms $\omega_{2,j}$ take the form 
\begin{equation}
\omega_{2,j} = \tilde p_j \Ai^2 + \tilde q_j \Ai'^2 + \tilde r_j \Ai\Ai',
\end{equation}
\end{subequations}
with polynomial coefficients $\tilde p_j,\tilde q_j,\tilde r_j\in \Q[s]$; the first few are
\[
\begin{gathered}
\tilde p_0(s) = -s,\quad \tilde q_0(s)=1, \quad \tilde r_0(s) = 0,\\*[2mm]
\tilde p_1(s) = -\frac{3s^2}{5}, \quad  \tilde q_1(s) = \frac{2s}{5}, \quad \tilde r_1(s) = \frac{3}{5},\\*[2mm]
\tilde p_2(s) = \frac{39s^3}{175} +\frac{9}{100}, \quad \tilde q_2(s) = -\frac{3s^2}{175}, \quad \tilde r_2(s) = -\frac{s^4}{25}-\frac{99s}{175}.
\end{gathered}
\]
\end{theorem}

For the Gaussian orthogonal and symplectic ensembles, we have the following theorem.

\begin{theorem}\label{thm:GOEGSE}
Recalling the definition \eqref{eq:nprime} for $n'$,
there holds, with $m$ being any fixed integer in the range $0\leq m\leq m_*$,\footnote{The bound $m_*$ is inherited from Theorem~\ref{thm:GUE}; see Footnote~\ref{fn:mstar} for its insignificance.} 
\begin{subequations}\label{eq:GOEGSE}
\begin{equation}
\begin{aligned}
\rho_{1,n,\infty}(x)\frac{dx}{ds}\bigg|_{x=\mu_{n',\infty} + \sigma_{n',\infty} s} &= \sum_{j=0}^m \omega_{1,j}(s) h_{n',\infty}^j + h_{n',\infty}^{m+1}\cdot O(e^{-2s}),\\*[2mm]
2\rho_{4,n,\infty}(x)\frac{dx}{ds}\bigg|_{x=\mu_{n',\infty} + \sigma_{n',\infty} s} &= \sum_{j=0}^m \omega_{4,j}(s) h_{n',\infty}^j + h_{n',\infty}^{m+1}\cdot O(e^{-2s}),
\end{aligned}\quad (n\to\infty),
\end{equation}
uniformly for $s$ bounded from below. Here, the expansion terms $\omega_{\beta,j}$ take the form 
\begin{equation}
\omega_{\beta,j} =  p_j\Ai^2 +  q_j \Ai'^2 +  r_j \Ai\Ai' +u_j \Ai\AI_\nu +v_j \Ai'\AI_\nu\Big|_{\nu = \llbracket \beta=1\rrbracket},
\end{equation}
\end{subequations}
with polynomial coefficients $p_j, q_j, r_j, u_j, v_j \in \Q[s]$ independent of $\beta$; the first few are
\begin{gather*}
 p_0(s) = -s,\quad  q_0(s)=1, \quad  r_0(s) = 0,\quad u_0(s) = \frac{1}{2},\quad v_0(s) = 0, \quad
 p_1(s) = -\frac{s^2}{2},\\*[2mm]  q_1(s) = \frac{2s}{5}, \quad  r_1(s) = \frac{3}{10}, \quad u_1(s) = -\frac{s}{10}, 
 \quad v_1(s) = \frac{s^2}{10}, \quad
 p_2(s) = \frac{3 s^3}{25}+\frac{279}{700}, \\*[2mm] q_2(s) = -\frac{27s^2}{350}, \quad  r_2(s) = -\frac{s^4}{100}-\frac{27 s}{140},\quad 
 u_2(s) = \frac{s^5}{100}+\frac{9 s^2}{140}, \quad v_2(s) = -\frac{3 s^3}{70}-\frac{9}{70}.
\end{gather*}
\end{theorem}
\begin{proof}
Taking \eqref{eq:rho14Intro} and \eqref{nprime14} into account, the result for $\beta=4$ follows immediately from the one for $\beta=1$. To prove the result for $\beta=1$, we write \eqref{eq:rho1} in the symmetrized form
\[
\rho_{1,n} = \frac{1}{2}\big(\rho_{2,n-1} + \rho_{2,n}\big) + \psi_{n-1} \left(1+ \frac{1}{2}\big(\Psi^\sharp_{n-2} + \Psi^\sharp_{n}\big)\right).
\]
By Theorem~\ref{thm:GUE}, $\rho_{2,n-1}dx/ds$ expands by using the parameters \eqref{eq:GUEscaling} indexed with $n-1$, whereas $\rho_{2,n}dx/ds$ expands by using those indexed with $n$. Thus, re-expanding using the parameters indexed at their center 
\[
n'=n-1/2 = \frac{1}{2}((n-1) +n)
\]
gives an expansion of 
\[
\frac{1}{2}\big(\rho_{2,n-1}(x) + \rho_{2,n}(x)\big)\frac{dx}{ds}\Big|_{x=\mu_{n'}+\sigma_{n'} s}
\]
in powers of $h_{n'}$ since the odd powers of $h_{n'}^{1/2}$ that appear when individually re-expanding $\rho_{2,n-1}(x)dx/ds$ and $\rho_{2,n}(x)dx/ds$  must cancel by symmetry. Similarly, by Theorem~\ref{thm:Hermite}, $\Psi^\sharp_{n-2}$ expands using the parameters indexed by $n-3/2$ and $\Psi^\sharp_{n}$ using those indexed by $n+1/2$. Thus, re-expanding using the parameters indexed at their center 
\[
n'=n-1/2 = \frac{1}{2}((n-3/2) + (n+1/2))
\]
gives an expansion of 
\[
\frac{1}{2}\big(\Psi^\sharp_{n-2} + \Psi^\sharp_{n}\big)\Big|_{x=\mu_{n'}+\sigma_{n'} s}
\]
in powers of $h_{n'}$, since odd powers of $h_{n'}^{1/2}$ must cancel by symmetry. By Theorem~\ref{thm:Hermite}, $\psi_{n-1}dx/ds$ already expands using the parameters indexed by $n'=n-1/2$. Thus, the general structure of the expansion \eqref{eq:GOEGSE} is proven. The specific first polynomials are obtained by a routine calculation using truncated power series.
\end{proof}

\subsection{Laguerre Ensembles}

Here, following \cite{B25-1}, the scaling and expansion parameters are 
\begin{equation}\label{eq:LUEscaling}
\begin{gathered}
\mu_{n,p} = \big(\sqrt{n}+\sqrt{p}\,\big)^2,\quad \sigma_{n,p}= \big(\sqrt{n}+\sqrt{p}\,\big) \bigg(\frac{1}{\sqrt{n}} + \frac{1}{\sqrt{p}}\bigg)^{1/3},\\*[2mm]
\tau_{n,p} = \frac{4}{\big(\sqrt{n}+\sqrt{p}\,\big) \left(\dfrac{1}{\sqrt{n}} + \dfrac{1}{\sqrt{p}}\right)},\quad 
h_{n,p} = \frac{\sigma_{n,p}}{\tau_{n,p}\mu_{n,p}}= \frac{1}{4} \bigg(\frac{1}{\sqrt{n}} + \frac{1}{\sqrt{p}}\bigg)^{4/3}.
\end{gathered}
\end{equation}
For $\beta=1,2,4$, we have $h_{n',p'} \asymp n^{-2/3}$ because of the bound $p>n-1$. We have $0< \tau_{n,p}\leq 1$, where $\tau_{n,p}=1$ corresponds to $n=p$, and $\tau_{n,p}\to 0^+$ as $p\to \infty$.

From the expansion \cite[Lemma~3.1]{B25-1} of the correlation kernel at the soft edge, we immediately get from \eqref{eq:rho2Kernel} the following theorem.\footnote{In a different parametrization, the case $m=1$ was previously studied in \cite[Eq.~(73)]{MR2178598}, \cite[Eq.~(3.34)]{MR2233711} for bounded $\alpha=p-n$, and in \cite[Eq.~(4.3)]{MR3802426} for $\alpha$ proportional to $n$.}

\begin{theorem}\label{thm:LUE}
For integer $n$ and real $p >n-1$, there holds, with~$m$ being any fixed integer in the range $0\leq m\leq m_*$,\footnote{See Footnote~\ref{fn:mstar} about the insignificance of $m_*$.}
\begin{subequations}\label{eq:LUE}
\begin{equation}
\rho_{2,n,p}(x)\frac{dx}{ds}\bigg|_{x=\mu_{n,p} + \sigma_{n,p} s} = \sum_{j=0}^m \omega_{2,j}(s;\tau_{n,p}) h_{n,p}^j + h_{n,p}^{m+1}\cdot O(e^{-2s}) \quad (n\to\infty),
\end{equation}
uniformly for $s$ bounded from below and $\tau_{n,p}$ bounded away from $0$. Here, the terms $\omega_{2,j}$ take the form 
(with the polynomials evaluated in $s,\tau$ and the Airy functions in $s$)
\begin{equation}
\omega_{2,j} = \tilde p_j \Ai^2 + \tilde q_j \Ai'^2 + \tilde r_j \Ai\Ai',\quad \tilde p_j,\tilde q_j,\tilde r_j\in \Q[s,\tau],
\end{equation}
\end{subequations}
with polynomial coefficients $\tilde p_j,\tilde q_j,\tilde r_j\in \Q[s,\tau]$; the first few are
\[
\begin{gathered}
\tilde p_0(s;\tau) = -s,\quad \tilde q_0(s;\tau)=1, \quad \tilde r_0(s;\tau) = 0,\\*[2mm]
\tilde p_1(s;\tau) = \frac{3(2 \tau-1)}{5}s^2, \quad \tilde q_1(s;\tau) = -\frac{2(2 \tau-1)}{5}s, \quad \tilde r_1(s;\tau) = \frac{3-\tau }{5},\\*[2mm]
\tilde p_2(s;\tau) = -\frac{214 \tau ^2-79 \tau -39}{175}s^3+ \frac{ (\tau -3)^2}{100}, \quad \tilde q_2(s;\tau) = \frac{143 \tau ^2-103 \tau -3}{175}s^2, \\*[2mm]
\tilde  r_2(s;\tau) = -\frac{(2 \tau -1)^2}{25}s^4 + \frac{29 \tau ^2-4 \tau -99}{175}s.
\end{gathered}
\]
\end{theorem}

For the Laguerre orthogonal and symplectic ensembles, we have the following theorem.

\begin{theorem}\label{thm:LOELSE}
Recalling the definition \eqref{eq:nprime} for of $n'$,
there holds, with $m$ being any fixed integer in the range $0\leq m\leq m_*$,\footnote{The bound $m_*$ is inherited from Theorem~\ref{thm:LUE}; see Footnote~\ref{fn:mstar} for its insignificance.} 
\begin{subequations}\label{eq:LOELSE}
\begin{equation}
\begin{aligned}
\rho_{1,n,p}(x)\frac{dx}{ds}\bigg|_{x=\mu_{n',p'} + \sigma_{n',p'} s} &= \sum_{j=0}^m \omega_{1,j}(s;\tau_{n',p'}) h_{n',p'}^j + h_{n',p'}^{m+1}\cdot O(e^{-2s}),\\*[2mm]
2\rho_{4,n,p}(x)\frac{dx}{ds}\bigg|_{x=\mu_{n',p'} + \sigma_{n',p'} s} &= \sum_{j=0}^m \omega_{4,j}(s;\tau_{n',p'}) h_{n',p'}^j + h_{n',p'}^{m+1}\cdot O(e^{-2s}),
\end{aligned}\quad (n\to\infty),
\end{equation}
uniformly for $s$ bounded from below and $\tau_{n',p'}$ bounded away from $0$. Here, the terms $\omega_{\beta,j}$ take the form
(with the polynomials evaluated in $s,\tau$ and the Airy functions in $s$) 
\begin{equation}
\omega_{\beta,j} =  p_j\Ai^2 +  q_j \Ai'^2 +  r_j \Ai\Ai' +u_j \Ai\AI_\nu +v_j \Ai'\AI_\nu\Big|_{\nu = \llbracket \beta=1\rrbracket},
\end{equation}
\end{subequations}
with polynomial coefficients $p_j, q_j, r_j, u_j, v_j \in \Q[s,\tau]$ independent of $\beta$; the first few are
\begin{gather*}
 p_0(s;\tau) = -s,\quad  q_0(s;\tau)=1, \quad  r_0(s;\tau) = 0,\quad u_0(s;\tau) = \frac{1}{2},\quad v_0(s;\tau) = 0,\\*[2mm]
 p_1(s;\tau) =  \frac{2\tau-1}{2}s^2, \quad   q_1(s;\tau) = -\frac{2(2\tau-1)}{5}s, \quad  r_1(s;\tau) = \frac{3-\tau}{10}, \quad u_1(s;\tau) = -\frac{3 \tau + 1}{10}s, \\*[2mm]
 \quad v_1(s;\tau) = -\frac{2\tau-1}{10}s^2, \quad
 p_2(s;\tau) = -\frac{51 \tau ^2-26 \tau -6}{50} s^3 + \frac{37 \tau ^2-372 \tau +558}{1400},\\*[2mm] 
 q_2(s;\tau) =\frac{272 \tau ^2-157 \tau -27}{350} s^2, \quad  r_2(s;\tau) = -\frac{(2 \tau -1)^2}{100}s^4 + \frac{13 \tau ^2-10 \tau -27}{140} s,\\*[2mm]
 u_2(s;\tau) = \frac{(2 \tau -1)^2}{100}s^5 +  \frac{33 \tau ^2+8 \tau +9}{140} s^2,\quad 
 v_2(s;\tau) = \frac{17 \tau ^2-5 \tau -3}{70} s^3+ \frac{\tau ^2+24 \tau -36}{280}.
\end{gather*}
\end{theorem}
\begin{proof}
Expect for using Theorems~\ref{thm:LUE} and \ref{thm:Laguerre} instead of Theorems~\ref{thm:GUE} and \ref{thm:Hermite}, the proof is literally the same as for Theorem~\ref{thm:GOEGSE}.
\end{proof}

\begin{remark}\label{rem:GELE}
Comparing the expansion coefficients in Theorems \ref{thm:GUE} and \ref{thm:GOEGSE} with those in Theorems~\ref{thm:LUE}  and \ref{thm:LOELSE}, we note that\footnote{We checked this up to $j=m_*$.}
\[
\begin{gathered}
\tilde p_j(t) = \tilde p_j(t;0), \quad \tilde q_j(t) = \tilde q_j(t;0), \quad \tilde r_j(t) = \tilde r_j(t;0),\\*[2mm]
 p_j(t) =  p_j(t;0), \quad  q_j(t) =  q_j(t;0), \quad  r_j(t) =  r_j(t;0), \quad  u_j(t) =  u_j(t;0), \quad  v_j(t) =  v_j(t;0).
\end{gathered}
\]
This is consistent with the scaling limits \eqref{eq:scalingLaw} as $p\to\infty$. However, since we did not establish the uniformity of the Laguerre expansions for all $0<\tau\leq 1$ (we conjecture this to be true) but only for $\tau$ being bounded away from zero (that is, for $\tau_0 \leq \tau \leq 1$ given any fixed $\tau_0>0$), applying the scaling limits, which are for fixed $n$, falls short of a proof.

\end{remark}

\section{Reconstruction of the Expansions of the Generating Functions}\label{sect:generating}

We continue the discussion started in Section~\ref{sect:GenTheory} and  show that the relation \eqref{eq:omegaAlt} allows us to reconstruct the polynomials \eqref{eq:GenPoly}. Because of Remark~\ref{rem:GELE}, it suffices to consider the Laguerre ensembles; the reconstruction for the Gaussian ensembles follows from setting $\tau=0$. 

The reconstruction mechanism generates three polynomial equations for $\beta=2$ but five polynomial equations for $\beta=1,4$. Hence, for a unique reconstruction, there are just enough equations for the first correction if $\beta=2$, but also for the second one if $\beta=1,4$.

\subsection{Unitary Ensembles} By omitting the indices for $\beta=2$ and $j=1$ in the notation, we get from \eqref{eq:LUE} and \eqref{eq:omegaAlt} the equation (derivatives are w.r.t. the variable $s$)
\[
\tilde p \Ai^2 + \tilde q \Ai'^2 + \tilde r \Ai \Ai' = P_1' \omega_0 + (P_1 + P_2')  \omega_0' + P_2 \omega_0'',
\]
where we infer from the Airy equation $\Ai'' = s \Ai$ that
\[
\omega_0 = \Ai'^2 - s\Ai^2, \quad \omega_0' = -\Ai^2,\quad \omega_0'' = -2 \Ai\Ai'.
\] 
Thus, the equation becomes
\begin{equation}\label{eq:beta2eq1}
\tilde p  \Ai^2 + \tilde q \Ai'^2 + \tilde r \Ai \Ai' = -(P_1 +s P_1' + P_2')\Ai^2 + P_1' \Ai'^2 - 2 P_2 \Ai\Ai'.
\end{equation}
It is shown in many places (see \cite{B25-2} and the literature cited therein) that the functions $\Ai,\Ai'$ are algebraically independent over the field of rational functions $\C(s)$. This means that $\Ai$, $\Ai'$ can be treated as further independent variables, so that \eqref{eq:beta2eq1} becomes a polynomial equation in $\Q[s,\tau,\Ai,\Ai']$, which by comparing coefficients is equivalent to a system of three polynomial equations in $\Q[s,\tau]$:
\[
\frac{3(2 \tau-1)}{5}s^2 = -P_1 -s P_1' - P_2', \quad -\frac{2(2 \tau-1)}{5}s = P_1', \quad \frac{3-\tau }{5} = -2P_2.
\]
Here, we took the concrete polynomials $\tilde p_1$, $\tilde q_1$, $\tilde r_1$ from Theorem~\ref{thm:LUE} on the left. Though looking superficially like a system of ordinary differential equations, it can be solved algebraically in the differential ring $\Q[s,\tau]$: using linear elimination for $P_2, P_1$ in that order -- consistently reinserting the derivatives of previously computed polynomials at each step -- we obtain the unique solution
\[
P_2 = \frac{\tau-3}{10}, \quad P_1 = -\frac{2\tau-1}{5}s^2, 
\] 
which reconstructs the polynomials displayed in \cite[Eq.~(3.7a)]{B25-1}.

\subsection{Orthogonal and Symplectic Ensembles} Because the cases $\beta=1$ and $\beta=4$ share the same polynomial 
coefficients in Theorem~\ref{thm:LOELSE} and in \eqref{eq:genTheory}, it suffices to consider the orthogonal case $\beta=1$.

\subsubsection{First correction terms} By omitting the indices for $\beta=1$ and $j=1$ in the notation, we get from \eqref{eq:LOELSE} and \eqref{eq:omegaAlt} the equation (writing $\AI=\AI_1$ for brevity; derivatives are w.r.t. the variable $s$)
\[
p \Ai^2 + q \Ai'^2 + r \Ai \Ai' + u \Ai \AI + v\Ai'\AI = P_1' \omega_0 + (P_1 + P_2')  \omega_0' + P_2 \omega_0'',
\]
where we infer from the Airy equation $\Ai'' = s \Ai$ and the relation $\AI'=\Ai$ that
\begin{equation}\label{eq:omega0Ders} 
\omega_0 = \Ai'^2 - s\Ai^2 + \frac{1}{2}\Ai\AI, \quad \omega_0' =\frac{1}{2}\big(\Ai'\AI - \Ai^2\big),\quad \omega_0'' = \frac{1}{2} \big(s \Ai \AI - \Ai\Ai'\big).
\end{equation}
Thus, the equation becomes
\begin{multline}\label{eq:beta1eq1}
p \Ai^2 + q \Ai'^2 + r \Ai \Ai' + u \Ai \AI + v\Ai'\AI \\*[2mm]
= -\frac12\big(P_1 + 2sP_1' + P_2'\big)\Ai^2 + P_1' \Ai'^2 - \frac12 P_2 \Ai\Ai' + \frac12\big(P_1' + sP_2\big)\Ai\AI + \frac12 \big(P_1 + P_2'\big)\Ai'\AI.
\end{multline}
By using tools from the Siegel--Shidlovskii theory of transcendental numbers and the differential Galois theory, we give in \cite[Theorem 1]{B25-2} two different proofs that, in fact, the three functions $\Ai,\Ai',\AI$ are algebraically independent over the field of rational functions $\C(s)$. This means that the three functions $\Ai$, $\Ai'$, $\AI$ can be treated as further independent variables, so that \eqref{eq:beta1eq1} becomes a polynomial equation in $\Q[s,\tau,\Ai,\Ai',\AI]$, which by comparing coefficients is equivalent to a system of five polynomial equations in $\Q[s,\tau]$,
\[
\begin{gathered}
\frac{2\tau-1}{2}s^2 = -\frac12\big(P_1 + 2sP_1' + P_2'\big), \quad -\frac{2(2\tau-1)}{5}s=P_1',\quad \frac{3-\tau}{10} = - \frac12 P_2, \\*[2mm]
\quad -\frac{3 \tau + 1}{10}s = \frac12\big(P_1' + sP_2\big),\quad  -\frac{2\tau-1}{10}s^2= 
\frac12 \big(P_1 + P_2'\big).
\end{gathered}
\]
Here, we took the concrete polynomials $p_1,q_1,r_1,u_1,v_1$ from Theorem~\ref{thm:LOELSE} on the left. By solving with linear elimination for $P_2, P_1$ in that order -- consistently reinserting the derivatives of previously computed polynomials at each step -- we obtain the unique solution
\[
P_2 =  \frac{\tau-3}{5}, \quad P_1 = -\frac{2\tau-1}{5}s^2,
\] 
which reconstructs the polynomials displayed in \cite[Eq.~(6.2a)]{B25-1}.

\subsubsection{Second correction terms} We proceed as before. By omitting the indices for $\beta=1$ and $j=2$ in the notation, we get from \eqref{eq:LOELSE} and \eqref{eq:omegaAlt} the equation (derivatives are w.r.t. the variable $s$)
\begin{multline*}
p \Ai^2 + q \Ai'^2 + r \Ai \Ai' + u \Ai \AI + v\Ai'\AI \\*[0mm]
= P_1' \omega_0 + (P_1 + P_2')  \omega_0' + (P_2 + P_3')  \omega_0'' + (P_3+ P_4')  \omega_0''' + P_4 \omega_0^{(4)},
\end{multline*}
where we have in addition to \eqref{eq:omega0Ders} that
\[
\omega_0''' =\frac{1}{2}\big(\Ai\AI+s\Ai'\AI-\Ai'^2\big),\quad \omega_0^{(4)} = \frac{1}{2} \big(\Ai^2-s \Ai \Ai'+s^2 \Ai \AI+ 2 \Ai'\AI \big).
\] 
Once again, comparing coefficients of the resulting polynomial equation in $\Q[s,\tau,\Ai,\Ai',\AI]$ gives a system of five polynomial equations in $\Q[s,\tau]$,
\[
\begin{aligned}
-\frac{51 \tau ^2-26 \tau -6}{50} s^3 + \frac{37 \tau ^2-372 \tau +558}{1400}&= -\frac12\big(P_1+2sP_1'+P_2'-P_4\big),\\*[2mm] \frac{272 \tau ^2-157 \tau -27}{350} s^2 &=\frac12\big(2P_1' -P_3 -P_4'\big),\\*[2mm] -\frac{(2 \tau -1)^2}{100}s^4 + \frac{13 \tau ^2-10 \tau -27}{140} s &= -\frac12\big(P_2 + P_3' +sP_4\big),\\*[2mm]
\frac{(2 \tau -1)^2}{100}s^5 +  \frac{33 \tau ^2+8 \tau +9}{140} s^2 &= \frac12\big(P_1' +s P_2 + P_3 + s P_3' + s^2 P_4 + P_4'\big),\\*[2mm] \frac{17 \tau ^2-5 \tau -3}{70} s^3+ \frac{\tau ^2+24 \tau -36}{280}&= \frac12\big(P_1 + P_2' + s  P_3 + 2P_4 + sP_4'\big).
\end{aligned}
\]
Here, we took the concrete polynomials $p_2,q_2,r_2,u_2,v_2$ from Theorem~\ref{thm:LOELSE} on the left. By solving with linear elimination for $P_4, P_3, P_2, P_1$ in that order -- consistently reinserting the derivatives of previously computed polynomials at each step -- we obtain the unique solution
\[
\begin{gathered}
P_4 = \frac{(\tau -3)^2}{50},\quad 
P_3 = -\frac{(\tau -3) (2 \tau -1)}{25} s^2,\quad
P_2 = \frac{(2 \tau -1)^2}{50} s^4 -\frac{2(4 \tau ^2+26 \tau -39)}{175} s,\\*[2mm]
P_1 = \frac{43 \tau ^2-18 \tau -8}{175} s^3+\frac{9 \tau ^2+496 \tau -744}{700},
\end{gathered}
\]
which reconstructs the polynomials displayed in \cite[Eq.~(6.2b)]{B25-1}.

Since only the final form \eqref{eq:genTheory} of the result concerning the generating function is assumed here -- without relying on any specific details of the general theory -- the unique reconstruction of the polynomial coefficients provides further compelling evidence supporting the conjectures underlying the cases $\beta = 1,4$ in the general theory presented in \cite{B25-3}, \cite{B25-1}.

\appendix

\section{Probabilistic Content of the Wave Functions}\label{app:probwave}

For the orthogonal ensembles, using the index consistency rule established in Section~\ref{sect:index}, we can rewrite the joint probability density \eqref{eq:joint} in terms of the 
function system \eqref{eq:chipsi} as
\begin{equation}\label{eq:pGOE}
P_{1,n}(x_1,\ldots,x_n) = \frac{1}{n!} \left|\det_{j,k=1}^n \psi_{j-1}(x_k)\right|.
\end{equation}
Here, the factor $1/n!$ comes from rescaling the Selberg integrals; see, e.g., \cite[Eqs.~(2.5.10/11)]{MR2760897}.

\begin{theorem}\label{thm:E} The system \eqref{eq:chipsi} of wave functions can be represented in the form  
\[
\frac{d}{dx} \E\bigg(\frac{1}{2}\sgn(x-x_1)\cdots  \sgn(x-x_n)\bigg) = 
\psi_{n-1}(x),
\]
where $x_1,\ldots,x_n$ are the levels of the orthogonal ensemble.
\end{theorem}

\begin{proof} Writing $F_n(x)$ for the expected value at hand, we get from \eqref{eq:pGOE} by symmetry
\[
F_n(x) = \frac{1}{2}\underset{x_1 < \cdots < x_n}{\int \cdots \int} \det_{j,k=1}^n \Big(\sgn(x-x_k)\psi_{j-1}(x_k)\Big)\,dx_1\cdots dx_n.
\]
By spotting here the operator $\epsilon$ as defined in \eqref{eq:epsilon} and using $(\epsilon f)'=f$, we calculate
\begin{align*}
F_n'(x) &= 
\sum_{j=1}^n \underset{x_1 < \cdots < x_{j-1} < x < x_{j+1} < \cdots < x_n}{\int \cdots \int}  \\*[0mm]
&\qquad{\small\begin{vmatrix}
\sgn(x-x_1) \psi_0(x_1) & \cdots &  \psi_0(x) &\cdots& \sgn(x-x_n) \psi_0(x_n) \\*[2mm]
\vdots & & \vdots & & \vdots \\*[2mm]
\sgn(x-x_1) \psi_{n-1}(x_1) & \cdots &  \psi_{n-1}(x) &\cdots& \sgn(x-x_n) \psi_{n-1}(x_n)
\end{vmatrix}}
\,dx_1 \cdots \widehat{dx_j} \cdots dx_n\\*[2mm]
&=  \sum_{j=1}^n \underset{x_1 < \cdots < x_{j-1} < x < x_{j+1} < \cdots < x_n}{\int \cdots \int} (-1)^{j-1} (-1)^{n-j}\\*[0mm]
&\qquad{\small\begin{vmatrix}
\psi_0(x) & \psi_0(x_1) & \cdots &  \widehat{\psi_0(x_j)} &\cdots& \psi_0(x_n) \\*[2mm]
\vdots &\vdots & & \vdots & & \vdots \\*[2mm]
\psi_{n-1}(x) & \psi_{n-1}(x_1) & \cdots &  \widehat{\psi_{n-1}(x_j)} &\cdots& \psi_{n-1}(x_n) 
\end{vmatrix}}
\,dx_1 \cdots \widehat{dx_j} \cdots dx_n\\*[2mm]
&= (-1)^{n-1} \underset{x_1 < \cdots < x_{n-1}}{\int \cdots \int}
{\small\begin{vmatrix}
\psi_0(x) & \psi_0(x_1) & \cdots & \psi_0(x_{n-1}) \\*[2mm]
\vdots &\vdots & \vdots  & \vdots \\*[2mm]
\psi_{n-1}(x) & \psi_{n-1}(x_1) &\cdots& \psi_{n-1}(x_{n-1}) 
\end{vmatrix}}
\,dx_1 \cdots  dx_{n-1}\\
\intertext{}
&=  \psi_{n-1}(x) \\*[0mm]
&\qquad +(-1)^{n-1}\sum_{j=0}^{n-2} (-1)^{j} \psi_j(x)\underset{x_1 < \cdots < x_{n-1}}{\int \cdots \int}
{\small\begin{vmatrix}
\psi_0(x_1) & \cdots & \psi_0(x_{n-1}) \\*[2mm]
\vdots &  & \vdots \\*[2mm]
\widehat{\psi_j(x_1)} & \cdots & \widehat{\psi_j(x_{n-1})} \\*[2mm]
\vdots &  & \vdots \\*[2mm]
\psi_{n-1}(x_1) & \cdots & \psi_{n-1}(x_{n-1})  
\end{vmatrix}}
\,dx_1 \cdots  dx_{n-1},
\end{align*}
where hats indicate missing objects.

According to an integration formula of de Bruijn \cite[p.~138]{MR79647}, the latter multiple integrals evaluate as Pfaffians,
\[
\underset{x_1 < \cdots < x_{n-1}}{\int \cdots \int}
{\small\begin{vmatrix}
\psi_0(x_1) & \cdots & \psi_0(x_{n-1}) \\*[2mm]
\vdots &  & \vdots \\*[2mm]
\widehat{\psi_j(x_1)} & \cdots & \widehat{\psi_j(x_{n-1})} \\*[2mm]
\vdots &  & \vdots \\*[2mm]
\psi_{n-1}(x_1) & \cdots & \psi_{n-1}(x_{n-1})  
\end{vmatrix}}
\,dx_1 \cdots  dx_{n-1} = \Pf(A_j) \quad (j=0,\ldots,n-2),
\]
where $A_j$ is obtained by deleting row $j$ and column $j$ from the following matrix $A=(a_{kl})$. If $n$ is odd, we have
(cf.~\cite[Eq.~(4.3)]{MR79647})
\[
a_{kl} = \langle \epsilon \psi^\sharp_k,\psi_l\rangle_2\quad (k,l=0,\ldots,n-1);
\]
whereas if $n$ is even, this skew-symmetric matrix gets bordered by (cf.~\cite[Eq.~(4.4)]{MR79647}, \eqref{eq:int})
\[
a_{k,n} = -a_{n,k} = \int_\Omega \psi_k(y)\,dy = \epsilon\psi^\sharp_k(\infty), \quad a_{nn} = 0.
\]
Theorem~\ref{thm:diffchi}, the bi-orthogonality \eqref{eq:biortho} and $\chi_l(\infty)=0$ imply that the rows in $A_j$ that come just from before and after the deleted one in $A$ are, in fact, equal. Hence, $\Pf(A_j)=0$ and we can thus infer
\[
F_n'(x) = \psi_{n-1}(x)
\]
as asserted.
\end{proof}

\begin{remark}
By observing  
\[
\int_\Omega\frac{d}{dx}\frac{1}{2}\E\big(\sgn(x-x_1)\cdots  \sgn(x-x_n)\big)\, dx = \frac{1-(-1)^n}{2} = \llbracket\text{$n$ is odd}\rrbracket,
\]
we see that Theorem~\ref{thm:E} implies Corollary~\ref{cor:int}.
\end{remark}

\begin{corollary}\label{cor:rho1} The level densities of the orthogonal and symplectic ensembles are related by
\begin{equation}\label{eq:rho1rho4}
\rho_{1,2n+1} = 2\rho_{4,n} + \psi_{2n}.
\end{equation}
\end{corollary}

\begin{proof} Using the index consistency rule of Section~\ref{sect:index}, we get from the Forrester--Rains interrelations \eqref{eq:ForresterRains} that the corresponding number distributions, as defined in \eqref{eq:density}, satisfy
\[
N_{1,2n+1}(x) = 2N_{4,n}(x) + \llbracket \text{$N_{1,2n+1}(x)$ is odd}\rrbracket.
\]
With $x_1,\ldots,x_{2n+1}$ denoting the levels of the orthogonal ensemble, we have
\[
\llbracket \text{$N_{1,2n+1}(x)$ is odd}\rrbracket = \frac{1}{2}\big(1 + \sgn(x-x_1)\cdots\sgn(x-x_{2n+1})\big).
\]
Taking expectations and then the derivative with respect to $x$ yields
\[
\rho_{1,2n+1}(x) = 2 \rho_{4,n}(x) + \frac{1}{2}\frac{d}{dx}\E\big(\sgn(x-x_1)\cdots\sgn(x-x_{2n+1})\big).
\]
Applying Theorem~\ref{thm:E} gives the assertion.
\end{proof}

\section{Asymptotic Expansions of the Wave Functions}\label{eq:PlancherelRotach}

The Poincaré-type expansion  of the ratio of two Gamma functions (see, e.g., \cite[Eq.~(C.4.3)]{MR1688958}) implies at once that
\begin{equation}\label{eq:gammaExpand}
\gamma_\nu \sim 2^{3/4} \nu^{-1/4} \left( 1-\frac{\nu^{-1}}{8} +\frac{\nu^{-2}}{128} + \frac{21\nu^{-3}}{1024} +\cdots \right) \quad (\nu\to\infty).
\end{equation}
With this at hand, the expansions of the Hermite and Laguerre functions $\phi_{n,\infty}$, $\phi_{n,p}$ given in \cite[§§10.1/10.3]{B25-1} can be straightforwardly re-expanded to establish those of $\psi_{n,\infty}$, $\psi_{n,p}$. 

\begin{theorem}\label{thm:Hermite}
With the parameters $\mu,\sigma,h$ denoting the corresponding quantities of \eqref{eq:GUEscaling} indexed at $n+1/2$, there holds, for any fixed integer $m$ as $n\to\infty$,
\begin{equation}\label{eq:expanHer2}
\left. \psi^\sharp_{n,\infty}(x) \frac{dx}{ds}\right|_{x=\mu+\sigma s}=   \Ai(s) + \sum_{k=1}^m \left({\mathfrak p}_k(s)\Ai(s) + {\mathfrak q}_k(s)\Ai'(s)\right)h^k + h^{m+1} O(e^{-s}),
\end{equation}
uniformly  for $s$ bounded from below, with certain ${\mathfrak p}_k, {\mathfrak q}_k \in \Q[s]$; the first of which are
\[
{\mathfrak p}_1(s) = -\frac{s}{5}, \quad {\mathfrak q}_1(s) = \frac{s^2}{5},\quad
{\mathfrak p}_2(s) = \frac{s^5}{50}+\frac{9 s^2}{70}, \quad {\mathfrak q}_2(s) = -\frac{3 s^3}{35}-\frac{9}{35}.
\]
If we denote the antiderivatives of $\psi^\sharp_{n,\infty}$ and $\Ai$ that vanish at $x=\infty$ by $\Psi^\sharp_{n,\infty}$ and $\AI_0$, the asymptotic expansion integrates to
\begin{equation}\label{eq:expanHerPrim}
\Psi^\sharp_{n,\infty}(x)\Big|_{x=\mu+\sigma s} =   \AI_0(s) + \sum_{k=1}^m \left({\mathfrak P}_k(s)\Ai(s) +  {\mathfrak Q}_k(s)\Ai'(s)\right)h^k + h^{m+1} O(e^{-s}),
\end{equation}
with certain ${\mathfrak P}_k, {\mathfrak Q}_k \in \Q[s]$; the first of which are
\[
{\mathfrak P}_1(s) = \frac{s^2}{5}, \quad  {\mathfrak Q}_1(s) = -\frac{3}{5},\quad 
{\mathfrak P}_2(s) = -\frac{29 s^3}{175}-\frac{309}{350},\quad  {\mathfrak Q}_2(s) = \frac{s^4}{50}+\frac{219 s}{350}.
\]
\end{theorem}

\begin{proof} From \cite[Corollary~10.1]{B25-1}, by observing $\sigma=h^{1/4}$ and re-expanding \eqref{eq:gammaExpand} as
\begin{equation}\label{gamman_as_h}
\gamma_n^{-1} \sim \frac{1}{2\sqrt{2} h^{3/8}} \left(1 + 2 h^3 + \cdots\right),
\end{equation}
we see that the expansion \eqref{eq:expanHer2} results from multiplying the expansion \cite[(10.11)]{B25-1} with the bracket in \eqref{gamman_as_h}. In particular, only the polynomials with index $k\geq 3$ in the original expansion \cite[Eq.~(10.12)]{B25-1} are getting transformed. Since the reminder term is integrable, 
\[
\Psi^\sharp_{n,\infty}(\mu +\sigma s) = -\sigma \int_{s}^\infty \psi^\sharp_{n,\infty}(\mu+\sigma t)\,dt
\]
is expanded by integrating \eqref{eq:expanHer2} termwise. To accomplish that, we proceed as follows. The Airy differential equation $s \Ai(s) = \Ai''(s)$ implies the replacement rule
\begin{equation}\label{eq:airyrule}
s^j \Ai^{(k)}(s) = s^{j-1} \Ai^{(k+2)}(s) - k s^{j-1} \Ai^{(k-1)}(s) \qquad (j\geq 1, \; k\geq 0)
\end{equation}
which, if repeated, allows us to absorb any polynomial prefactors into higher-order derivatives of $\Ai$. Hence, the right-hand-side of \eqref{eq:expanHer2} becomes
{\small\[
\Ai(s) +  \left(\frac{\Ai^{(5)}(s)}{5} -\Ai''(s)\right)h +\left(\frac{\Ai^{(10)}(s)}{50} - \frac{17\Ai^{(7)}(s)}{35} +\frac{5\Ai^{(4)}(s)}{2} -2 \Ai'(s)\right)h^2 + \cdots,
\]}%
where that $\Ai(s)$ does only appear in the leading term (this follows from going through the construction underlying the proof of \cite[Corollary~10.1]{B25-1}). The resulting expansion is now easily integrated to \eqref{eq:expanHerPrim}.
\end{proof}

\begin{theorem}\label{thm:Laguerre}
With the parameters $\mu,\sigma,h,\tau$ denoting the corresponding quantities of \eqref{eq:LUEscaling} indexed at $n+1/2, p+1/2$,
there holds, for any fixed non-negative integer $m$ as $n\to\infty$,
\begin{equation}\label{eq:expanLag4}
\left.\psi^\sharp_{n,p}(x) \frac{dx}{ds}\right|_{x=\mu+\sigma s} =   \Ai(s) + \sum_{k=1}^m \left({\mathfrak p}_k(s;\tau)\Ai(s) + {\mathfrak q}_k(s;\tau)\Ai'(s)\right)h^k + h^{m+1} O(e^{-s}),
\end{equation}
uniformly for $s$ bounded from below and $\tau$ bounded away from $0$. The coefficients are certain polynomials ${\mathfrak p}_k, {\mathfrak q}_k \in \Q[s,\tau]$; the first few are
\[
\begin{gathered}
{\mathfrak p}_1(\tau,s) = -\frac{3\tau +1}{5}s, \quad
{\mathfrak q}_1(\tau,s) = -\frac{2 \tau -1}{5}s^2,\\*[1mm]
{\mathfrak p}_2(\tau,s) = \frac{(2 \tau -1)^2 }{50}s^5 + \frac{33 \tau ^2+8 \tau +9}{70}s^2, \quad
{\mathfrak q}_2(\tau,s) = \frac{17 \tau ^2-5 \tau -3}{35}s^3 + \frac{\tau ^2+24 \tau -36}{140}.
\end{gathered}
\]
If we denote the antiderivatives of $\psi^\sharp_{n,p}$ and $\Ai$ that vanish at $x=\infty$ by $\Psi^\sharp_{n,p}$ and $\AI_0$, the asymptotic expansion integrates to
\begin{equation}\label{eq:expanLagPrim}
\Psi^\sharp_{n,p}(x)\Big|_{x=\mu+\sigma s} =   \AI_0(s) + \sum_{k=1}^m \left({\mathfrak P}_k(s;\tau)\Ai(s) + {\mathfrak Q}_k(s;\tau)\Ai'(s)\right)h^k + h^{m+1} O(e^{-s}),
\end{equation}
with certain polynomial coefficients ${\mathfrak P}_k, {\mathfrak Q}_k \in \Q[s,\tau]$; the first few are
\[
\begin{gathered}
{\mathfrak P}_1(\tau,s) = -\frac{2 \tau -1}{5}s^2, \quad
{\mathfrak Q}_1(\tau,s) = \frac{\tau -3}{5}, \\*[1mm]
{\mathfrak P}_2(\tau,s) = \frac{29 \tau ^2+31 \tau -29}{175}s^3 + \frac{23 \tau ^2+412 \tau -618}{700}, \\*[1mm]
{\mathfrak Q}_2(\tau,s) = \frac{(2 \tau -1)^2 }{50}s^4 -\frac{9 \tau ^2+146 \tau -219}{350}s.
\end{gathered}
\]
\end{theorem}
\begin{proof} From \cite[Corollary~10.3]{B25-1}, by observing $\sigma=2\tau^{-1}h^{-1/2}$, re-expanding \eqref{eq:gammaExpand}
as 
\begin{equation}\label{eq:gammaprod}
\gamma_n^{-1}\gamma_p^{-1} \sim \frac{1}{4} h^{-3/4} \tau^{-1/2} \Big(1+\frac{\tau ^2-8 \tau +8}{4}h^3 + \cdots\Big),
\end{equation}
and expanding
\begin{equation}\label{eq:sqrtexpand}
\left(\frac{\mu+\sigma s}{2}\right)^{-1/2} = \tau h^{3/4} \Big(1- \frac{1}{2} \tau s h + \frac{3}{8}(\tau s h)^2  - \frac{5}{16}   (\tau s  h)^3 + \cdots\Big)
\end{equation}
we see that the expansion \eqref{eq:expanLag4} results from multiplying the expansion \cite[(10.28)]{B25-1} with the brackets in \eqref{eq:gammaprod} and \eqref{eq:sqrtexpand}. The antiderivative is dealt with as in Theorem~\ref{thm:Hermite}.
\end{proof}

\begin{remark}
If we compare the expansions of the Hermite and Laguerre functions in Theorems \ref{thm:Hermite} and \ref{thm:Laguerre}, we observe the following simple relations between the corresponding polynomial coefficients:\footnote{We checked this up to $j=m_*$.}
\[
{\mathfrak p}_j(t) = {\mathfrak p}_j(t;0),\quad {\mathfrak q}_j(t) = {\mathfrak q}_j(t;0), \quad {\mathfrak P}_j(t) = {\mathfrak P}_j(t;0), \quad {\mathfrak Q}_j(t) = {\mathfrak Q}_j(t;0).
\]
This is consistent with the scaling limits \eqref{eq:scalingLaw} as $p\to\infty$. However, since we did not establish the uniformity of the Laguerre expansions for all $0<\tau\leq 1$ (we conjecture this to be true) but only for $\tau$ being bounded away from zero (that is, for $\tau_0 \leq \tau \leq 1$ given any fixed $\tau_0>0$), applying the scaling limits, which are for fixed $n$, falls short of a proof. 

\end{remark}

\section{The Laguerre-to-Hermite Scaling Limit}\label{app:ScalingLaw}

Using the scaling and expansion parameters \eqref{eq:GUEscaling} and \eqref{eq:LUEscaling}, we easily get from \eqref{eq:gammaExpand} that
\[
\lim_{p\to\infty}\frac{\sigma_{n,p}^{1/2}}{\gamma_p} \left(\frac{\mu_{n,p}+\sigma_{n,p} s}{2}\right)^{-1/2} = \sigma_{n,\infty}^{1/2},\quad \lim_{p\to\infty} h_{n,p}=h_{n,\infty},\quad \lim_{p\to\infty} \tau_{n,p} = 0.
\]
In \cite[Remark~10.5]{B25-1}, we proved that
\[
\lim_{p\to\infty} \sigma_{n,p}^{1/2} \phi_{n,p} (\mu_{n,p} + \sigma_{n,p} s) = \sigma_{n,\infty}^{1/2} \phi_{n,\infty} (\mu_{n,\infty} + \sigma_{n,\infty} s).
\]
This implies the following scaling limits for the bi-orthogonal system of Laguerre functions, as introduced in Section~\ref{app:funcSkew}:
\begin{subequations}\label{eq:scalingLaw}
\begin{equation}\label{eq:chipsiScalingLaw}
\begin{aligned}
\lim_{p\to\infty}  \psi_{n,p} (x)\, dx \Big|_{x=\mu_{n,p} + \sigma_{n,p} s} &= \psi_{n,\infty} (x)\,dx\Big|_{x=\mu_{n,\infty} + \sigma_{n,\infty} s},\\*[2mm]
\lim_{p\to\infty}  \chi_{n,p} (x) \Big|_{x=\mu_{n,p} + \sigma_{n,p} s} &= \chi_{n,\infty} (x) \Big|_{x=\mu_{n,\infty} + \sigma_{n,\infty} s};
\end{aligned}
\end{equation}
where we note that the $\psi_{n,p}$ transform as densities and the $\chi_{n,p}$ as functions. 
The limits hold uniformly for $s$ bounded from below and can be differentiated w.r.t. $s$. Hence, integration and differentiation give
\begin{equation}\label{eq:chipsiScalingLaw2}
\begin{aligned}
\lim_{p\to\infty}  \epsilon\psi_{n,p} (x) \Big|_{x=\mu_{n,p} + \sigma_{n,p} s} &= \epsilon \psi_{n,\infty} (x)\Big|_{x=\mu_{n,\infty} + \sigma_{n,\infty} s},\\*[2mm]
\lim_{p\to\infty}  \chi_{n,p}' (x)\,dx \Big|_{x=\mu_{n,p} + \sigma_{n,p} s} &= \chi_{n,\infty}' (x)\,dx \Big|_{x=\mu_{n,\infty} + \sigma_{n,\infty} s}.
\end{aligned}
\end{equation}
An application of \eqref{eq:rho2}, \eqref{eq:rho4} implies
\begin{equation}\label{eq:chipsiScalingLaw3}
\lim_{p\to\infty}  \rho_{\beta,n,p} (x)\,dx \Big|_{x=\mu_{n,p} + \sigma_{n,p} s} = \rho_{\beta,n,\infty} (x)\,dx \Big|_{x=\mu_{n,\infty} + \sigma_{n,\infty} s} \quad (\beta=1,2).
\end{equation}
\end{subequations}
By Corollary~\ref{cor:rho1}, the scaling limit of the level density $\rho_{4,n,p}$ takes the same form except that we have to replace $n,p$ with $2n,2p$  in the scaling parameters $\mu$ and $\sigma$. This, and \eqref{eq:chipsiScalingLaw3}, follows also from the general $\LbE_{n,p} \to \GbE_{n}$ transition law\footnote{For the $\LOE$ with integer $p$ this transition law  is implied by the multivariate central limit theorem: namely, if a matrix $X_{n,p}$ is drawn from the standard $n$-variate Wishart distribution with $p$ degrees of freedom, one gets in distribution (cf. \cite[Theorem 3.4.4]{MR1990662})
\[
\frac{X_{n,p}-pI}{\sqrt{2p}} \to \GOE_n \quad (p\to \infty).
\]
By continuity, the limit can be lifted to the eigenvalue distributions; see, e.g., \cite[pp.~125--126]{MR145620}, \cite[Corollary~13.3.2]{MR1990662}, and \cite[Corollary~9.5.7]{MR652932}.} \cite[Theorem~B.1]{B25-1}, \cite[Theorem~1(ii)]{MR3413957} by integration.

\let\oldaddcontentsline\addcontentsline
\renewcommand{\addcontentsline}[3]{}%


\let\addcontentsline\oldaddcontentsline
\addtocontents{toc}{\vspace{1\baselineskip}}
\bibliographystyle{spmpsci}
\bibliography{paper}

\end{document}